\documentclass[11pt]{article} 
\usepackage{amssymb}
\usepackage{amsthm,amsmath,amssymb}
\usepackage{amsthm,amsmath}

\font\cmrfootnote=cmr10 scaled 750
\font\cmrone=cmr10 scaled \magstep1
\font\cmrtwo=cmr10 scaled \magstep2
\font\cmrthree=cmr10 scaled \magstep3
\font\cmrfour=cmr10 scaled \magstep4
\font\cmrfive=cmr10 scaled \magstep5

\font\cmrscfootnote=cmr7 scaled 750
\font\cmrscone=cmr7 scaled \magstep1
\font\cmrsctwo=cmr7 scaled \magstep2
\font\cmrscthree=cmr7 scaled \magstep3
\font\cmrscfour=cmr7 scaled \magstep4
\font\cmrscfive=cmr7 scaled \magstep5

\font\cmrscscfootnote=cmr5 scaled 750
\font\cmrscscone=cmr5 scaled \magstep1
\font\cmrscsctwo=cmr5 scaled \magstep2
\font\cmrscscthree=cmr5 scaled \magstep3
\font\cmrscscfour=cmr5 scaled \magstep4
\font\cmrscscfive=cmr5 scaled \magstep5

\font\mitfootnote=cmmi10 scaled 750
\font\mitone=cmmi10 scaled \magstep1
\font\mittwo=cmmi10 scaled \magstep2
\font\mitthree=cmmi10 scaled \magstep3
\font\mitfour=cmmi10 scaled \magstep4
\font\mitfive=cmmi10 scaled \magstep5

\font\mitscfootnote=cmmi7 scaled 750
\font\mitscone=cmmi7 scaled \magstep1
\font\mitsctwo=cmmi7 scaled \magstep2
\font\mitscthree=cmmi7 scaled \magstep3
\font\mitscfour=cmmi7 scaled \magstep4
\font\mitscfive=cmmi7 scaled \magstep5

\font\mitscscfootnote=cmmi5 scaled 750
\font\mitscscone=cmmi5 scaled \magstep1
\font\mitscsctwo=cmmi5 scaled \magstep2
\font\mitscscthree=cmmi5 scaled \magstep3
\font\mitscscfour=cmmi5 scaled \magstep4
\font\mitscscfive=cmmi5 scaled \magstep5

\font\cmsyfootnote=cmsy10 scaled 750
\font\cmsyone=cmsy10 scaled \magstep1
\font\cmsytwo=cmsy10 scaled \magstep2
\font\cmsythree=cmsy10 scaled \magstep3
\font\cmsyfour=cmsy10 scaled \magstep4
\font\cmsyfive=cmsy10 scaled \magstep5

\font\cmsyscfootnote=cmsy7 scaled 750
\font\cmsyscone=cmsy7 scaled \magstep1
\font\cmsysctwo=cmsy7 scaled \magstep2
\font\cmsyscthree=cmsy7 scaled \magstep3
\font\cmsyscfour=cmsy7 scaled \magstep4
\font\cmsyscfive=cmsy7 scaled \magstep5

\font\cmsyscscfootnote=cmsy5 scaled 750
\font\cmsyscscone=cmsy5 scaled \magstep1
\font\cmsyscsctwo=cmsy5 scaled \magstep2
\font\cmsyscscthree=cmsy5 scaled \magstep3
\font\cmsyscscfour=cmsy5 scaled \magstep4
\font\cmsyscscfive=cmsy5 scaled \magstep5

\font\cmexfootnote=cmex10 scaled 750
\font\cmexone=cmex10 scaled \magstep1
\font\cmextwo=cmex10 scaled \magstep2
\font\cmexthree=cmex10 scaled \magstep3
\font\cmexfour=cmex10 scaled \magstep4
\font\cmexfive=cmex10 scaled \magstep5

\font\cmexscfootnote=cmex10 scaled 750
\font\cmexscone=cmex10 scaled \magstep1
\font\cmexsctwo=cmex10 scaled \magstep2
\font\cmexscthree=cmex10 scaled \magstep3
\font\cmexscfour=cmex10 scaled \magstep4
\font\cmexscfive=cmex10 scaled \magstep5

\font\cmexscscfootnote=cmex7 scaled 750
\font\cmexscscone=cmex7 scaled \magstephalf       
\font\cmexscsctwo=cmex7 scaled \magstep1
\font\cmexscscthree=cmex7 scaled \magstep2
\font\cmexscscfour=cmex7 scaled \magstep3
\font\cmexscscfive=cmex7 scaled \magstep4

\def\mathfootnote{\textfont0=\cmrfootnote \textfont1=\mitfootnote \textfont2=\cmsyfootnote \textfont3=\cmexfootnote
        \scriptfont0=\cmrscfootnote \scriptscriptfont0=\cmrscscfootnote \scriptfont1=\mitscfootnote \scriptscriptfont1=\mitscscfootnote
        \scriptfont2=\cmsyscfootnote \scriptscriptfont2=\cmsyscscfootnote \scriptfont3=\cmexscfootnote \scriptscriptfont3=\cmexscscfootnote}

\def\mathone{\textfont0=\cmrone \textfont1=\mitone \textfont2=\cmsyone \textfont3=\cmexone
        \scriptfont0=\cmrscone \scriptscriptfont0=\cmrscscone \scriptfont1=\mitscone \scriptscriptfont1=\mitscscone
        \scriptfont2=\cmsyscone \scriptscriptfont2=\cmsyscscone \scriptfont3=\cmexscone \scriptscriptfont3=\cmexscscone}
\def\mathtwo{\textfont0=\cmrtwo \textfont1=\mittwo \textfont2=\cmsytwo \textfont3=\cmextwo
        \scriptfont0=\cmrsctwo \scriptscriptfont0=\cmrscsctwo \scriptfont1=\mitsctwo \scriptscriptfont1=\mitscsctwo
        \scriptfont2=\cmsysctwo \scriptscriptfont2=\cmsyscsctwo \scriptfont3=\cmexsctwo \scriptscriptfont3=\cmexscsctwo}
\def\maththree{\textfont0=\cmrthree \textfont1=\mitthree \textfont2=\cmsythree \textfont3=\cmexthree
        \scriptfont0=\cmrscthree \scriptscriptfont0=\cmrscscthree \scriptfont1=\mitscthree \scriptscriptfont1=\mitscscthree
   \scriptfont2=\cmsyscthree \scriptscriptfont2=\cmsyscscthree \scriptfont3=\cmexscthree \scriptscriptfont3=\cmexscscthree}
\def\mathfour{\textfont0=\cmrfour \textfont1=\mitfour \textfont2=\cmsyfour \textfont3=\cmexfour
        \scriptfont0=\cmrscfour \scriptscriptfont0=\cmrscscfour \scriptfont1=\mitscfour \scriptscriptfont1=\mitscscfour
        \scriptfont2=\cmsyscfour \scriptscriptfont2=\cmsyscscfour \scriptfont3=\cmexscfour \scriptscriptfont3=\cmexscscfour}
\def\mathfive{\textfont0=\cmrfive \textfont1=\mitfive \textfont2=\cmsyfive \textfont3=\cmexfive
        \scriptfont0=\cmrscfive \scriptscriptfont0=\cmrscscfive \scriptfont1=\mitscfive \scriptscriptfont1=\mitscscfive
        \scriptfont2=\cmsyscfive \scriptscriptfont2=\cmsyscscfive \scriptfont3=\cmexscfive \scriptscriptfont3=\cmexscscfive}

\def\1m{\mathone}
\def\2m{\mathtwo}
\def\3m{\maththree}
\def\4m{\mathfour}
\def\5m{\mathfive}

\font\rmfoot=cmr8

\usepackage{amscd, amssymb}
\usepackage{hyperref}
\usepackage{comment}
\usepackage{enumitem}

\newtheorem{theorem}{Theorem}[section]
\newtheorem{lemma}[theorem]{Lemma}
\newtheorem{prop}[theorem]{Proposition}
\newtheorem{corollary}[theorem]{Corollary}
\newtheorem{cor}[theorem]{Corollary}
\newtheorem{conjecture}[theorem]{{Conjecture}}

\newtheorem{claim}[theorem]{{Claim}}

\theoremstyle{remark}
\newtheorem{remark}[theorem]{Remark}

\theoremstyle{definition}

\newtheorem{definition}[theorem]{{Definition}}

\newtheorem{exer}[theorem]{Exercise}

\def\bclaim{\begin{claim}}
\def\eclaim{\end{claim}}
\def\bdefin{\begin{definition}}
\def\edefin{\end{definition}}
\def\bcor{\begin{corollary}}
\def\ecor{\end{corollary}}
\def\bthm{\begin{theorem}}
\def\ethm{\end{theorem}}
\def\bconj{\begin{conjecture}}
\def\econj{\end{conjecture}}
\def\blem{\begin{lemma}}
\def\elem{\end{lemma}}
\def\blemma{\begin{lemma}}
\def\elemma{\end{lemma}}
\def\bprop{\begin{prop}}
\def\eprop{\end{prop}}
\def\bremark{\begin{remark}}
\def\eremark{\end{remark}}

\newcommand{\RR}{\mathbb R}

\newcommand{\NN}{\mathbb N}
\newcommand{\ZZ}{\mathbb Z}

\newcommand{\TT}{\mathbb T}

\DeclareMathOperator{\id}{Id}

\newcommand{\ra}{\right\rangle}

\def\max{{\operatorname{max}}}

\def\q{\quad}
\def\qq{\qquad}

\def\bpf{\begin{proof}}
\def\epf{\end{proof}}
\def\beq{\begin{equation}}
\def\eeq{\end{equation}}
\def\beqno{\begin{equation*}}
\def\eeqno{\end{equation*}}
\def\eaeq{\end{aligned}}
\def\baeq{\begin{aligned}}

\def\del{\partial}

\def\h#1{\hbox{#1}}

\def\wt{\widetilde}
\def\lb{\label}

\def\th{\theta}

\def\del{\partial}

\def\ra{\rightarrow}
\def\eps{\epsilon}
\def\del{\partial}

\def\w{\wedge}

\def\beq{\begin{equation}}
\def\eeq{\end{equation}}
\def\bi#1{\bibitem{#1}}

\def\h#1{\hbox{#1}}

\def\calM{{\mathcal M}}

\def\calX{{\mathcal X}}
\def\calA{{\mathcal A}}
\def\calY{{\mathcal Y}}
\def\calB{{\mathcal B}}

\def\gr{{\operatorname{gr}}}

\def\supp{\h{\rm supp}}

\def\bs{\bigskip}
\def\MA{\operatorname{MA}}

\def\be{\beta}
\def\per{\operatorname{per}}

\def\strop{\operatorname{strop}}

\def\tsper{\per_{\strop}}
\def\supp{\operatorname{supp}}

\def\mathoverr#1#2{\buildrel #1 \over #2}
\def\Ent{\h{\rm Ent}}
\def\pstar{p^\star}
\def\ul#1{\underline{#1}}
\def\Cvx{\operatorname{Cvx}}
\def\bu{$\bullet\,\,$}

\def\MAeq{Monge--Amp\`ere equation }
\def\MAeqs{Monge--Amp\`ere equations }

\font\itnotsosml=cmti7

\def\thhit#1{\hbox{${\hbox{#1}}^{\,{\hbox{\itnotsosml th}}}$}}

\begin{document}
\title{On large deviation principles and the Monge--Amp\`ere equation 
(following Berman, Hultgren)}
\author{Yanir A. Rubinstein}
\date{8 June 2022}

\maketitle

\centerline{\it Dedicated to Steve Zelditch on the occasion of his 
6\thhit{8} birthday  }

\begin{abstract}
This is mostly an exposition, aimed to be accessible to geometers,
analysts, and probabilists, of a fundamental recent theorem of R. Berman
with recent developments by J. Hultgren,
that asserts that the second boundary value problem 
for the real Monge--Amp\`ere equation admits a
probabilistic interpretation, in terms of many particle limit 
of permanental point processes satisfying a large deviation principle with a rate function given explicitly using optimal transport.
An alternative proof of a step in the Berman--Hultgren Theorem
is presented allowing to to deal with all ``tempratures" simultaneously 
instead of first reducing to the zero-temperature case.

\end{abstract}

\section{Introduction}

The purpose of this exposition
is to present one particularly beautiful connection between the 
\MAeq and probability, specifically, a large deviation principle,
discovered by Berman. 
Since 
the original work by Berman is still unpublished \cite{Berm},
and moreover deals with the more technically involved
case where the gradient image is a polytope (that arises from 
toric varieties), it
seemed more pedagogical to give an exposition that
concentrates on subsequent work of Hultgren \cite{Hult} 
that elaborates Berman's ideas in the case the gradient image
has no boundary (that arises from Abelian varieties) as 
many of the key ideas are present already in the latter setting. 

It is worth pointing out that while the lack of boundary is a simplification,
Hultgren beautifully deals with a different set of technicalities that
arises in the Abelian setting that is absent from Berman's toric 
setting: theta function analysis. 

We completely skip the connection
to Abelian varieties in this exposition as our goal was to present
strictly the Monge--Amp\`ere to LDP connection, stripping away the
underlying geometry.
Our exposition culminates in Theorem \ref{hultthm}, due to Hultgren.

We take the opportunity to present an alternative proof 
to Theorem \ref{hultthm} that deals with all ``tempratures" simultaneously instead
of first reducing to the zero-temperature case as in the work of Berman and Hultgren
(cf. \cite[Remark 24, p. 59]{HultThesis}).
Basically this amounts to replacing an application of the 
G\"artner--Ellis theorem with a direct computation (we still use
G\"artner--Ellis theorem in several other places).
The proof we present culminates in \S\ref{FinDimWasSubSec}
and is  self-contained in the sense that we present essentially all the basic 
prerequisites from large deviations theory and optimal transport. 
 
The family of \MAeqs Berman originally considers actually corresponds to
and is inspired by the Ricci continuity method introduced by the author
in 2008 \cite{R08} in connection with the Ricci flow and the search for 
K\"ahler--Einstein metrics.
The idea there, explained in detail in the survey \cite[\S6]{R}, is to extend
Aubin's continuity method originally defined for parameter values $t\in[0,1]$
all the way `back' in time to $t\in(-\infty,1]$. This is motivated
by the Ricci flow \cite[\S3]{R08},\cite[\S6]{R} and is exploited heavily
in subsequent work on existence of singular K\"ahler--Einstein metrics where the standard continuity
method of Aubin cannot be readily used, but where the asymptotic analysis
of the limit $t\ra-\infty$ allows to bypass the difficulty in getting the continuity method `started' \cite[\S9]{JMR}. Berman \cite{Berm} discovered a physics interpretation for this analytical gadget where the temperature corresponds 
precisely to $-1/t$, and so the limit $t\ra-\infty$ becomes for him 
a `zero-temperature limit'. Making this connection to physics proved extremely fruitful
as it led Berman to several observations, including the LDP result we describe in this survey.

\bigskip
\noindent
{\bf Goal of present work.} 
The purpose of these lectures is to give a detailed exposition of
some of Berman's ideas \cite{Berm} in the setting of Hultgren's work \cite{Hult}
hopefully with some simplification (in particular the alternative proof
mentioned above). 
We give  some additional background in probability, hopefully to
allow the dissemination of this beautiful piece of mathematics to
a wider audience, given that the necessary background from probability
might not be standard for most students in geometric analysis.
We learned the little probability that we were able to present here from reading
Berman and Hultgren \cite{Berm,Hult} as well as using
the classic reference of Dembo--Zeitouni \cite{DemboZ} and 
the more recent textbook of Rassoul-Agha--Sepp\"al\"ainen \cite{ras}
where thorough, and probably more accurate, presentations of the results in 
\S\ref{LDPSec}--\S\ref{LDPwithoutSec} can be found. For the results on optimal
transport our main reference is Ambrosio--Gigli \cite{AmbGig}.

\bigskip
\noindent
{\bf Organization.} 
We start by giving some anecdotes from the history
of relations between \MAeqs and probability in \S\ref{MAprobSec}. 
This by no means is even an attempt at an exhaustive historical
account. Rather, it is for the sake of placing the idea of Berman in 
a broader historical context. Section \ref{LDPSec} serves as
a gentle crash (oxymoron alert) course on large deviation principles (LDP).
Section \ref{MomentSec} discusses moment generating functions and
the associated LDP with rate function coming from the Legendre transform
of the function: this is the G\"artner--Ellis Theorem.
Section \ref{GESec} completes the proof of the G\"artner--Ellis Theorem.
Section \ref{LDPwithoutSec} discusses a general criterion for the existence
of an LDP without using a moment generating function.
Section \ref{OTSec} briefly reviews the fundamentals of optimal transportation,
and computes, following Berman \cite{Berm}, the Legendre transform of Wasserstein distance as well
as identifies the candidate rate function for a family of Monge--Amp\`ere
equations \eqref{maeq} related to the Ricci continuity method \cite{R08,R}.
Section \ref{MomentMASec} presents the proof of Berman--Hultgren's Theorem
\ref{hultthm}, showing an LDP for a sequence of empirical measures
arising from theta functions on Abelian varieties (although we do not
go into any of the underlying complex geometry, which is beautifully
presented in Hultgren's work and in fact is one of the novelties
of his work \cite{Hult}). The rate function is related to optimal
transport, and the whole construction is intimately related to solutions
of the ``master equation" \eqref{maeq}. 
Most of Section \ref{MomentMASec} is devoted to 
our approach described above to the proof of Theorem
\ref{hultthm}, culminating in  \S\ref{FinDimWasSubSec},
and in \S\ref{AltProofSubSec} we present the original proof
of Berman and Hultgren, and briefly compare the two approaches.

\section{Monge--Amp\`ere and probability}
\label{MAprobSec}

The Laplace and Poisson equations have myriad 
probabilistic connections and interpretations, e.g., through Brownian motion, eigenfunctions, nodal sets \cite{Chung, KS96, Zelditch}.
Being the higher-dimensional analogue of these equations, one
would expect similar, albeit more complicated, relations between
the (homogeneous or non-homogeneous) \MAeq and probability.

Perhaps Gaveau was the first to pioneer such relations, when he 
discovered that the solution to the complex \MAeq can be
expressed as the value function of a stochastic
optimal control problem and found a semi-group that
can be studied in relation to a parabolic version of 
the \MAeq  \cite{Gav1,Gav2,Gav3}.
This generalized the classical probabilistic representation of the
solution of the Laplace equation in one complex dimension.
Another fundamental relation was discovered by Krylov
who proved $C^{1,1}$ a priori estimates for the real Monge--Amp\`ere
equation, among other results
 \cite{Kry89,Kry87,Kry89b,Kry90}.
We refer to Delarue \cite{Delarue} for excellent lecture notes 
that survey, expand, and give a pedagogical point of view on both
Gaveau's and Krylov's achievements
(and also cover the complex case for the latter).
Another type of relation between Monge--Amp\`ere and probability
arises in the theory of optimal transportation
(see, e.g., Villani \cite{Villani,Villani:oldandnew})
where one seeks
a map pushing forward one probability measure into another. 
Indeed, the optimal (cost-minimizing) 
map can be expressed as the gradient of a convex function and the
push-forward equation becomes a real Monge--Amp\`ere equation.

Another classical connection appears through the theory of Markov semigroups,
which in turn are closely related to the heat kernel.
In this context hypercontractivity plays a central role and
leads to (logarithmic and regular) Sobolev inequalities. 
This goes back to work of Gross \cite{Gross}. For our anecdotal storytelling 
we mention that Bakry and Bakry--Ledoux \cite{Bakry,BakryLedoux} 
showed how to use these
ideas to establish Sobolev and diameter estimates in the presence
of positive Ricci curvature. This was then applied in the setting
of a degenerate complex \MAeq 
 \cite[Proposition 6.2]{JMR} to by-pass standard ``Riemannian" proofs that 
do not readily apply in the degenerate setting.

A spectacular relation between \MAeqs and probability was discovered
by Zelditch who together with collaborators studied several instances where large deviation principles (LDP) make their appearance in complex geometry
\cite{ZeitZeld1,Zeld2,FengZeld}. In particular, Song--Zelditch found
a large deviation principle that underlies the canonical 
Bergman approximation scheme 
\cite{Don,PS}
for 
the Cauchy problem for the homogeneous real \MAeq
(this equation governs initial value geodesics in the space of K\"ahler metrics 
with toric symmetry) \cite{SongZelditch}.

Berman subsequently discovered that an LDP holds
also in the quite distinct setting of the second 
boundary value problem for non-homogeneous real \MAeq \cite{Berm}
that appears naturally in the setting of toric K\"ahler manifolds
as well as optimal transport. Since the gradient image in this
setting is a polytope, there are issues with the corners and the boundary
that render the computations more involved.
For that reason, the subsequent follow-up work of Hultgren is more
appropriate for our exposition, as in the setting of Abelian varieties
that he studies the gradient image is a torus, while the main
features of Berman's work are still present. The sections of the 
line bundles over Abelian varieties, theta functions, 
are more complicated to represent than the simple
toric monomials that appear in the case of toric
varieties, but that is not a steep price to pay for
the lack of boundary.
Although outside the scope of these notes, we mention that
Berman also discovered an LDP in a sort of 
complex version of his aforementioned toric result in the case $\beta>0$,
where the toric variety is replaced by a polarized complex manifold and the role of the permanental point process is played by a determinantal point process
\cite{BermanCMP}.

\section{Large deviation principles}
\label{LDPSec}

We will be interested in asymptotic behavior, or more precisely, the asymptotic concentration, of a sequence of probability spaces
$$
(\calX,\calA,\mu_n),
$$
indexed by $n\in\NN$.
Here,
$(\calX,\calA)$ is
a measure space, i.e., $\calX$ is a set, also called the sample space, consisting
of all possible outcomes, and $\calA$ is a $\sigma$-algebra (the collection of measurable sets in $\calX$), and
$\{\mu_n\}$ are probability measures on $(\calX,\calA)$, i.e., functions $\mu_n:\calA\ra[0, \infty ]$ satisfying $\mu_n (\emptyset) = 0$ and $\mu_n (\cup_i A_i)=\sum_i \mu_n ( A_i)$ whenever $A_k\cap A_l=\emptyset$ for all $k,l$, and with total mass $\mu_n (\calX ) = 1$. The last property is what makes a general measure a probability measure.
Often, it is customary to omit $\calA $ from the notation and refer to the triple 
$(\calX,\calA,\mu_n)$ simply by
the notation $\mu_n
\in
 P(\calX)$,  where $P(\calX)$ denotes the space of probability measures on $(\calX,\calA)$.

\begin{definition}
\label{LDPDef}
We say that 
$
\{(\calX,\calA,\mu_n)\}_n
$
(or, for brevity, sometimes just $\{\mu_n
\}_n) $ satisfies a large deviation principle with normalization  $r_n\nearrow\infty$ (and denote this statement by LDP$(\mu_n, r_n)$) if
 there exists a lower semicontinuous function $I:\calX \ra [0, \infty ] $ such that
$$
\liminf_{n\ra\infty}\frac1{r_n} \log \mu_n (O)\ge -\inf_O I, \qq \forall\; O \h{\ \ open in \ } \calX,  
$$
and
$$
\limsup_{n\ra\infty}\frac1{r_n} \log \mu_n (C)\le -\inf_C I, \qq \forall\; C \h{\ \ closed in \ } \calX  .
$$\end{definition}

Under mild assumptions on $\calX$, 
there is actually no ambiguity in the {\it rate function} $I$ \cite[Theorem 2.13]{ras}.
Here is where the stipulation that the rate function be lower semicontinuous is relevant.

\begin{lemma}
\label{ILemma}
Suppose $\calX  $ is a regular topological space. Then,
$$
I(x)=\sup\Big\{-\liminf\frac1{r_n}\log \mu_n(O): O\ni x, O \h{\ \ is an open set in  } \calX\Big\}.
$$
\end{lemma}

\begin{remark} {\rm
\label{}
For a sort of converse see Proposition \ref{ldpprop}. 
} \end{remark}

\begin{remark} {\rm
\label{}
The definition of a regular topological space will
be given in the proof shortly. 
} \end{remark}

\begin{proof}
Define $I$ by the above formula.
Suppose that LDP$(\mu_n, r_n)$ holds with rate function $F$. 
By Definition \ref{LDPDef}, whenever $O \h{\ \ is an open set in  } \calX$ with  $O\ni x$,
$$
F (x)\ge \inf_O F\ge 
-\liminf\frac1{r_n}\log \mu_n(O).
$$
Taking the supremum over all such $O$ does not change the left
hand side, while the right hand side becomes $I(x)$. Thus,
$F\ge I$.

Conversely, fix $x $ and let $c$ be such that $c<F(x)$. It suffices to show that 
$c\le I(x) $. Indeed, the assumption on $\calX $ means that any point can be separated from any closed set not containing it by means of disjoint open sets. Thus, we can separate $x $ from the closed set $\{F\le c\}\not\ni x$, i.e., choose open $G\ni x$ with
$\overline {G}\cap \{F\le c\}=\emptyset$, i.e., 
$\overline {G}\subset \{F> c\}$. 
Now,
since $F$ is lower semicontinuous its inf over any closed set is attained.
In particular, $\inf_{\overline {G}} F\ge c$. (Note that we could not otherwise conclude
$c \le \inf_{\{F> c\}} F$!)
Thus, $$
c
\le \inf_{\overline {G}} F\le -\limsup\frac1{r_n}\log 
\mu_n (\overline {G})\le -\liminf\frac1{r_n}\log \mu_n ({G}) \le I (x),
$$
so $F(x)\le I(x)$, concluding the proof.
\end{proof}

The fact that the function $I $ is nonnegative is crucial: it means probability of events (an event is an element of $\calA$) is exponentially decaying in general,
with {\it rate} 
$$
\h{
`$r_n\times$ infimum of $I$ over the closure of the event'.
}
$$ 
Of particular interest are therefore the zeros of the rate function, i.e., the events $I^{-1}(0)$ (when this set is non-empty; it is always non-empty if
the rate function is {\it good} \cite[p. 4]{DemboZ}, i.e., $I$ has
compact sub-level sets in $\calX$).
The significance of zeros is nicely captured in terms of random variables, which we now turn to discuss.
\begin{remark} {\rm
\label{zeroremark}
Note that by setting $O = C =\calX  $ it follows that $\inf I=0$.
} \end{remark}

\bs
\centerline{$\star$\hglue1cm$\star$\hglue1cm$\star$}
\bs

Probability measures and random variables can often be interchanged in the discussion, and by abuse of terminology this will sometimes be the case. Let us briefly discuss the terminology involved.
Let $(\calX,\calA,\mu)$ be a probability space and let $(\calY,\calB)$ be a measure space (a typical example is $\RR$ with $\calB$
being the usual Borel sets).
A  {\it random variable } with values in $\calY $ is a measurable 
function $X:\calX \ra  \calY$ (here, measurable takes into account both $\calA $ and $\calB $).
To such an $X$ one may associate a probability space
$
(\calY,\calB,\nu)
$
defined by $$\nu (B): =\mu(\{x\in\calX\,:\,X(x)\in B\}), \h{\ \ for $B\in\calB$}
$$
(the previous formula is often written, with some abuse,
as $\nu (B): =\mu\{X\in B\}$). 
In other words, $\nu (B):=\mu(X^{-1}(B))$, or,
$$
\nu = X_{\#}\mu.
$$
One also refers to $\nu$ as the {\it law  } of $X$.
Note that often when discussing the random variable $X$ the ``background space" 
$(\calX,\calA,\mu)$ is completely auxiliary/irrelevant since one is completely focused on $(\calY,\calB,\nu)$.
Thus, often one does not distinguish between $X$ and 
 the resulting or pushed-forward measure $\nu$.

\bs
\centerline{$\star$\hglue1cm$\star$\hglue1cm$\star$}
\bs

When the sequence of probability spaces
$$
(\calX,\calA,\mu_n)
$$
happen to
correspond to the laws of a sequence of random variables $X_n$ (on some auxiliary probability space $(\widetilde\calX,\widetilde\calA,\widetilde\mu)$) with values in \hbox{$\calX $}, the zeros of the rate function have the following meaning. We say that
$x\in \hbox{$\calX $ } $ is a limiting value of
$\{X_n\}$ with probability $c$ if  
$$
\lim_n\widetilde\mu (\{|X_n-x|<\eps\})
=
\lim_n\mu_n(B_\eps(x))=
c,
$$ for all $\eps>0$ small. 
 
\begin{lemma}
\label{Lemma}
Consider a sequence of random variables $X_n$ whose laws $\mu_n$ satisfy
LDP$(\mu_n, r_n)$.
Suppose that $x\in \hbox{$\calX $ } $ is a limiting value of
$\{X_n\}$ with probability $c>0$ (independent of $n$). Then $x\in I^{-1}(0)$.
In fact, $x\in I^{-1}(0)$ whenever $\frac1{r_n}\log\mu_n (\{|X_n-x|<\eps\})\ra 0$.
\end{lemma}

In other words, $x\in I^{-1}(0)$ whenever 
$x\in \hbox{$\calX $ } $ is a ``limiting value of
$X_n$ with probability decaying slower than $e^{-Cr_n}$ (for some $C>0$)."

\begin{proof}
One has
$$
 \inf_{\overline {B_\eps(x)}} I
\le -\limsup\frac1{r_n}\log \mu_n(X_n\in \overline {B_\eps(x)})
= -\limsup\frac1{r_n}\log \widetilde\mu_n(X_n^{-1}(\overline {B_\eps(x)}))
=0.$$
Letting $\epsilon \ra  0$ and using $I\ge0$ and lower semicontinuity guarantees $I(x)=0$ since $I\ge0$.
\end{proof}

One particular situation of practical interest is when $I^{-1}(0)$ is a singleton.
In that case a sort of converse of the previous statement holds.

\begin{corollary}
\label{weakconvcor}
Suppose that the sequence of probability spaces
$
(\calX,\calA,\mu_n)
$ satisfies
LDP$(\mu_n, r_n)$ with a good rate function $I$ (i.e., $I$ has
compact sub-level sets in $\calX$) and that $I^{-1}(0)=\{x\}
\subset\hbox{$\calX $}$. Then $\mu_n\ra \delta_x$ weakly.

\end{corollary}

\begin{proof}
 First, note that:

\begin{claim}
\label{limsupclaim}
For all closed $C\subset \calX$,
$$
\limsup\mu_n (C)\le \delta_x (C).
$$
\end{claim}

\begin{proof}
Indeed, this is trivial when $x\in C$ since then  the right-hand side equals 1. Otherwise,
$$
\limsup_{n\ra\infty}\frac1{r_n} \log \mu_n (C)\le -\inf_C I=:-\eps<0,
$$
since $I$ is good 
and so the infimum is attained, and by assumption it cannot be zero as $x$ is the only zero of the nonnegative function $I$. Thus, for $n$ large
$$
\frac1{r_n} \log \mu_n (C)\le -\inf_C I=-\eps,
$$
so $\mu_n (C)< e^{-r_n\eps}$ and recalling
that $r_n\nearrow\infty$ (Definition \ref{LDPDef}), 
$\limsup\mu_n (C)=0=\delta_x(C)$ as desired.
\end{proof}
The proof now follows from the so-called portmanteau theorem, but we give the quick proof for completeness. Note that Claim \ref{limsupclaim} implies 
$$
\liminf\mu_n (O)\ge \delta_x (O),
$$
for all open $O\subset \calX$ (by looking at $C= \hbox{$\calX $ } \setminus  O$).
To show the convergence is equivalent to showing
$$
\lim \int_{\calX} f\mu_n=\int_{\calX} f\delta_x,
$$
for all bounded continuous functions $f$.  Indeed, if $a\le f\le b$,
\begin{equation*}
\begin{aligned}
\label{}
\liminf \int_{\calX} f\mu_n
&=
\liminf \int_a^b \mu_n\{f>t\}dt
\cr
&\ge
 \int_a^b \liminf\mu_n\{f>t\}dt
\cr
&\ge
 \int_a^b \delta_x\{f>t\}dt
\cr
&=\int_{\calX} f\delta_x.
\end{aligned}
\end{equation*}
Repeating the above computation for $-f$ gives
$\limsup \int_{\calX} f\mu_n\le \int_{\calX} f\delta_x$,
so Corollary \ref{weakconvcor}  follows.
\end{proof}

\section{Moment generating functions}
\label{MomentSec}

The logarithm of the moment generating function of the sequence of probability spaces
$
(\calX,\calA,\mu_n)
$ with normalization $\{r_n\}$ is defined by
$$
p (\theta): =\lim \frac1{r_n}\log \int^{}_{ \calX  } e^{r_n\langle \theta , x \rangle }\mu_n (x),
$$
assuming the limit exists and is finite, for each $\theta \in \hbox{$\calX^* $}$. 
(For example, if $ \hbox{$\calX $ } = \RR^n $ then $ \hbox{$\calX^* $ } = \RR^n$, and 
if $ \hbox{$\calX $ } = P(\RR^d) $ then $ \hbox{$\calX^* $ } = C^0(\RR^d)$.)

A generating function encodes a lot of information, as discovered by G\"artner, and rediscovered by Ellis.
The following theorem holds for rather general  \hbox{$\calX $ } but we will prove it for $ \hbox{$\calX $ } = \RR^n $ although the proof essentially works verbatim for any locally convex topological vector space  \hbox{$\calX $ } under the assumption of exponential tightness of the sequence of measures
(defined in \S\ref{ExpTightSubSec}).  
Being a novice in the field, the author will follow  
Berman and Zelditch and refer to the next theorem as
the G\"artner--Ellis Theorem, although a more accurate attribution
of credit is given in the historical notes in Dembo--Zeitouni to which
the reader is warmly referred to for much more accurate statements of 
all the results on LDPs that we discuss in these notes \cite[\S2.3]{DemboZ}.

In general, exponential tightness is a necessary assumption (which is automatic for  $\calX = \hbox{$\RR^n$}  $ (see   Lemma \ref{exptighlemma}) and $\calX = P(X)$ for $X$ a compact manifold). In the following we denote by $f^\star$ the 
Legendre dual of $f$ \cite[p. 104]{Rock},
\begin{equation}
\label{LegDefEq}
f^\star(y):=\sup_x[\langle x,y\rangle -f(x)].
\end{equation}
(We will later use this same definition more generally
for functions on abstract spaces where the pairing will be 
taken to be the natural one in each setting.)

\begin{theorem}
\label{GEThm}
Suppose that the moment generating function $p$ of the sequence of probability spaces
$
(\calX,\calA,\mu_n)
$ with normalization $\{r_n\}$ is well defined and Gateaux differentiable. Then
LDP$(\mu_n, r_n)$ with rate function $p^\star$.
\end{theorem}

Before proving the  G\"artner--Ellis Theorem let us prove two famous corollaries thereof.

\subsection{Cram\'er's Theorem}
\label{}

Let $\{X_i\}_{i = 1}^n $ be independent, identically distributed, random variables (i.i.d.r.v.) on $\wt\calX$ with values in \hbox{$\RR$}. This means that the law of $X_i $  is equal to some $\mu\in P(\RR)$ regardless of $i $. The {\it sample mean } is by definition the random variable with values in \hbox{$\RR$},
$$
S_n:=\sum X_i/n.
$$
This is the ``probability" notation. Recalling that a random variable is really a function leads to a more precise notation. The random variable $S_n$ is the measurable function $S_n:\wt\calX^n\ra\RR$
$$
(x_1,\ldots,x_n)\mapsto[X_1(x_1)+\ldots X_n(x_n)]/n.
$$
This is really the composition
$$
(x_1,\ldots,x_n)\mapsto(X_1(x_1),\ldots,X_n(x_n))\mapsto[X_1(x_1)+\ldots X_n(x_n)]/n,
$$
and so the law of $S_n $  is the push forward of the law of the  \hbox{$\RR^n$}-valued random variable 
$$X_1\otimes\cdots\otimes X_n:=(X_1,\ldots,X_n)$$ under the ``mean map"
$$
s_n:(k_1,\ldots,k_n)\mapsto \sum k_i/n.
$$ 
The law of $X_1\otimes\cdots\otimes X_n$ is $\mu\otimes\cdots\otimes \mu$.
Thus, the law of $S_n $ is $\mu_n: =(s_n)_{\#}\mu^{\otimes n}$.

\begin{corollary}
\label{CramerCor}
Suppose that $\mu=fdx$ with $f\in C^0(\RR)$ and with compact support.
LDP$(\mu_n, n)$ with rate function $p^\star*$.
\end{corollary}

\begin{proof}
Set $\theta:=(\theta_1,\ldots ,\theta_n)\in \hbox{$\RR^n$}$.
The moment generating function is
\begin{equation}
\begin{aligned}
\label{}
p (\theta) &=\lim\frac 1n\log\int^{}_{\RR}e^{n x\theta}\mu_n(x)  
\cr
&=
\lim\frac 1n\log
\int^{}_{ \RR^n}e^{n\langle s_n(a_1,\ldots,a_n)\theta}\mu^{\otimes n}(a_1,\ldots,a_n)  
\cr
&=
\lim\frac 1n\log
\int^{}_{ \RR }e^{a_1\theta} \mu(a_1)\cdots
\int^{}_{ \RR }e^{a_n\theta} \mu(a_n)
\cr
&=\lim\frac 1n\log
\Big(\int^{}_{\RR }e^{a\theta} \mu(a)^n\Big)^n
\cr
&=\log
\int^{}_{ \RR }e^{a\theta} \mu(a)
\end{aligned}
\end{equation}
This is $C^1$ because of the assumption on $\mu $ so we are done by Theorem \ref{GEThm}. 
\end{proof}

\subsection{Sanov's theorem}
\label{}

The projection via the mean map gives rather crude information. 
Set $X^n:=X\otimes\cdots\otimes X$.
Another sequence of measures that can be obtained from the $n$-fold product via
the ``empirical" map $\delta^n:X^n\ra P(X)$,
\begin{equation*}
\begin{aligned}
\label{}
\delta^n (x_1,\ldots , x_n): =\frac1{n}\sum_{i = 1}^n\delta_{x_i}.
\end{aligned}
\end{equation*}
The measures
\begin{equation}
\begin{aligned}
\label{}
\Gamma_n:=(\delta_n)_{\#}\mu^{\otimes n}\in P(X),
\end{aligned}
\end{equation}
all live on the same space and therefore can be studied via a large deviation principle, if the associated  moment generating function exists.
Note that $\Gamma_n$ is the law of the random variable $\delta^n:(X^n
,\mu^{\otimes n})\ra P(X),$ i.e., of the  {\it random measure } $\delta^n$
(where the randomness is determined by $\mu^{\otimes n}$, i.e., by sampling $n$ points in $X$ independently, each according to $\mu$).

Define the entropy functional $\Ent:P(X)\times P(X)\ra \RR$,
\begin{equation}
\begin{aligned}
\label{Enteq}
\Ent(\mu,\nu):=\int_X \log\frac{\nu}{\mu}\nu,
\end{aligned}
\end{equation}
whenever $\nu$ is absolutely continuous with respect to $\mu$,
and $\infty$ otherwise.

\begin{corollary}
\label{SanovCor}
Suppose that $\mu=fdx$ with $f\in C^0(X)$ and with compact support.
LDP$(\Gamma_n, n)$ with rate function $\Ent(\mu,\,\cdot\,)$.
\end{corollary}

\begin{proof}
Now  \hbox{$\calX = P(X)$. } 
Let $\theta\in C^0_b(X)=\calX^*$.
The moment generating function is
\begin{equation}
\begin{aligned}
\label{}
p (\theta) &=\lim\frac 1n
\log\int^{}_{P(X)}e^{n \langle \theta,\nu \rangle }\Gamma_n(\nu)  
\cr
&=
\lim\frac 1n
\log\int^{}_{P(X)}e^{n \langle \theta,\nu \rangle }(\delta_n)_{\#}\mu^{\otimes n}
\cr
&=
\lim\frac 1n
\log\int^{}_{X^n}e^{n \langle \theta,\delta^n(x_1,\ldots,x_n)\rangle }
\mu(x_1)\otimes\cdots\mu(x_n)
\cr
&=
\lim\frac 1n
\log\int^{}_{X^n}e^{n \langle \theta,\frac1n\sum_{i = 1}^n\delta_{x_i}\rangle }
\mu(x_1)\otimes\cdots\mu(x_n)
\cr
&=
\lim\frac 1n
\log\int^{}_{X^n}e^{\sum_{i = 1}^n\theta(x_i) }
\mu(x_1)\otimes\cdots\mu(x_n)
\cr
&=
\lim\frac 1n
\log\Big(\int^{}_{X}e^{\theta}\mu\Big)^n
\cr
&=
\log\int^{}_{X}e^{\theta}\mu.
\end{aligned}
\end{equation}
This is $C^1$ because of the assumption on $\mu $ so we are done by Theorem \ref{GEThm}
and Lemma \ref{Entlemma} below. 
\end{proof}

Define $I: C^0(X)\ra  \RR$,
\begin{equation}
\begin{aligned}
\label{}
I_\mu(\theta):=
\log\int_{X}e^{\theta}\mu.
\end{aligned}
\end{equation}
Recall the definition of the Legendre transform \eqref{LegDefEq},
where in the following lemma the pairing is taken to be the usual
``integration pairing''' between functions and measures.

\begin{lemma}
\label{Entlemma}
The Legendre transform of $\Ent(\mu,\,\cdot\,)$ is $I_\mu$ 
and vice versa.
\end{lemma}

\begin{proof}
First, $I_\mu$ is convex on $C(X)$ since it is a moment generating function (see Lemma  \ref{mgfcvxlemma}). Alternatively, the arguments in the proof of Lemma \ref{mgfcvxlemma} show convexity. We will show that Legendre transform  of
$I_\mu$ is $\Ent(\mu,\,\cdot\,)$ which therefore will imply that the latter is also convex. We claim that
\begin{equation}
\begin{aligned}
\label{entIeq}
\Ent(\mu,\nu)+I_\mu(\theta)\ge\langle \theta,\nu\rangle. 
\end{aligned}
\end{equation}

Indeed,
\begin{equation*}
\begin{aligned}
\label{}
I_\mu(\theta)-\langle \theta,\nu\rangle
&=
\log\int_{X}e^{\theta}\mu-\langle \theta,\nu\rangle
\cr
&=
\log\int_{X}e^{\theta}\frac\mu\nu \nu-\langle \theta,\nu\rangle
\cr
&\ge
\int_{X}\log\Big(e^{\theta}\frac\mu\nu\Big) \nu-\langle \theta,\nu\rangle
\cr
&=
\int_{X}\Big(\theta +\log\frac\mu\nu\Big) \nu-\langle \theta,\nu\rangle
=-\Ent(\mu,\nu),
\end{aligned}
\end{equation*}
with equality if and only if $\nu=e^\th\mu/\int e^\th\mu$ (so that $\nu\in P(X)$). Thus, 
$\Ent(\mu,\nu)\ge I_\mu^\star(\nu):=\sup_{\th}[\langle \theta,\nu \rangle- 
I_\mu(\th)]$. On the other hand, putting $\th=\log\frac\nu\mu$,
$$
\Big\langle \log\frac\nu\mu,\nu \Big\rangle- 
I_\mu\Big(\log\frac\nu\mu\Big)
=
\Ent(\mu,\nu)-0,
$$
so
$\Ent(\mu,\nu)\le \sup_{\th}[\langle \theta,\nu \rangle- 
I_\mu(\th)]$. Thus,
$\Ent(\mu,\nu)= I_\mu^\star(\nu)$ and so in particular from
the general property of the Legendre transform \eqref{LegDefEq}
(see also \cite[p. 104]{Rock}),
\begin{equation}
\label{LegBasicProp}
f(x)+f^\star(y)\ge\langle x,y\rangle,
\end{equation}
it follows that 
\eqref{entIeq} holds, as claimed.
Equation \eqref{entIeq} gives,
$$
\sup_{\nu}\big[\langle\th ,\nu \rangle- \Ent(\mu,\nu)\big]
\le 
I_\mu(\th),
$$
and now putting $\nu=e^\th\mu/\int e^\th\mu$  we see equality is attained,
so $I_\mu$ is the Legendre transform of $\Ent(\mu,\,\cdot\,)$, concluding
the proof.
\end{proof}

\subsection{Properties of the moment generating function}

\begin{lemma}
\label{mgfcvxlemma}
The moment generating function is convex.
\end{lemma}

\begin{proof}
The pointwise limit of a sequence of  convex functions is convex (one way to think about it is in terms of the epigraphs--- and clearly the limits of convex sets is a convex sets, and the limits of epigraphs is moreover an epigraph). Thus, it suffices to show that 
$$
\frac1{r_n}\log \int^{}_{ \hbox{$\calX $ } } e^{r_n\langle \theta , x \rangle }\mu_n (x)
$$
is a convex function. Indeed, since $||fg||_{L^1(\mu)}\le ||f||_{L^2(\mu)}||g||_{L^2(\mu)}$,
\begin{equation*}
\begin{aligned}
\label{}
\frac1{r_n}\log \int^{}_{ \hbox{$\calX $ } } e^{r_n\langle (\theta_1+\theta_2)/2 , x \rangle }\mu_n (x)
&=
\frac1{r_n}\log \int^{}_{ \hbox{$\calX $ } } 
\sqrt{e^{r_n\langle \theta_1 , x \rangle }}
\sqrt{e^{r_n\langle \theta_2 , x \rangle }}
\mu_n (x)
\cr
&\le 
\frac1{r_n}\log 
\Bigg(
\sqrt{\int^{}_{ \hbox{$\calX $ } } e^{r_n\langle \theta_1 , x \rangle }\mu_n (x)}
\sqrt{\int^{}_{ \hbox{$\calX $ } } e^{r_n\langle \theta_2 , x \rangle }\mu_n (x)}
\Bigg)
\cr
&= 
\frac12
\Bigg(
\frac1{r_n}\log 
{\int^{}_{ \hbox{$\calX $ } } e^{r_n\langle \theta_1 , x \rangle }\mu_n (x)}
+
\frac1{r_n}\log 
{\int^{}_{ \hbox{$\calX $ } } e^{r_n\langle \theta_2 , x \rangle }\mu_n (x)}
\Bigg),
\end{aligned}
\end{equation*}
as desired.
\end{proof}

\begin{lemma}
\label{}
$p^\star$ is convex and nonnegative.
\end{lemma}

\begin{proof}
By definition
$$
 \pstar (x): = \sup_{\theta }[\langle x,\theta  \rangle - p (\theta )]
$$
is a supremum of affine functions, hence it is convex. Plugging in $\theta = 0$ and using that  $p(0)=0$ it follows that $\pstar (x)\ge 0$.
\end{proof}

The reader that compares the statement of Theorem \ref{GEThm} to that in some books might note that one does not really need to assume the moment generating function is differentiable and certain weaker assumptions are enough. One of them though is automatic from convexity: 

\begin{lemma}
\label{}
If there exists a small ball $B$ about the origin on which $p< \infty 
$
then $p> - \infty 
$ everywhere.
\end{lemma}

\begin{proof}
This is a general fact about convex functions that can be proved as follows:
$$
p (0)\le \frac1{1+C} p (-C\theta ) + \frac C{1+C} p (\theta ),
$$ so
$$
 \frac C{1+C} p (\theta )
\ge 
p (0) - \frac1{1+C} p (-C\theta ). 
$$
Now, choose $C >0$ small enough so that $C\theta \in B$ and note $p(0)=0$.
\end{proof}

\section{Proof of the G\"artner--Ellis Theorem  }
\lb{GESec}

The goal of this section 
is to give a proof of  Theorem \ref{GEThm}.
The proof of the upper bound (the one about close sets) is a little easier 
and so we will go over it at first. There are two main steps: first, prove the upper bound for compact sets; second, show that compact sets capture the general case.

\subsection{The upper bound for compact sets}
\lb{uppbdsubsec}
The upper bound for compact sets is sometimes called the weak upper bound LDP. Let $C\subset \hbox{$\calX $ }  $ then be a compact set. We claim that
$$
\limsup_{n\ra\infty}\frac1{r_n} \log \mu_n (C)\le -\inf_C  \pstar .
$$

Fix $\delta>0$.
For each $x\in \hbox{$\calX $ } $ 
let $y(x)\in  \hbox{$\calX^* $}$
satisfy
$\langle x,y(x)\rangle-p(y(x))>\pstar(x)-\delta$ 
and let $B_x^{\delta,y(x)}$ be a neighborhood of 
$x\in \hbox{$\calX $ } $ defined as follows
$$
B_x^{\delta,y(x)}:=\{z\in \hbox{$\calX $ } \,:\, |\langle z ,y(x) \rangle -
\langle x,y(x) \rangle  | < \delta\}.
$$
Finitely many 
neighborhoods of affine subspaces
$B_{x_1}^{\delta,y_1(x_1)},\ldots, B_{x_m}^{\delta,y_m(x_m)}$ cover $C$ by 
compactness.
Observe that asymptotically we reduce the calculations for the ``left-hand side" to one ball:
\begin{equation*}
\begin{aligned}
\label{}
\limsup_{n\ra\infty}\frac1{r_n} \log \mu_n (C)
&\le 
\limsup_{n\ra\infty}\frac1{r_n} \log \sum_{i = 1}^m\mu_n (B_{x_i}^{\delta,y_i(x_i)})
\cr
&= 
\limsup_{n\ra\infty}\frac1{r_n} \log\big( 
m\sup_{i\in\{1,\ldots m\}}\mu_n (B_{x_i}^{\delta,y_i(x_i)})
\big)\cr
&= 
\limsup_{n\ra\infty}\bigg[\frac1{r_n} \log 
m + \frac1{r_n}\log\sup_{i\in\{1,\ldots m\}}\mu_n (B_{x_i}^{\delta,y_i(x_i)})
\bigg]\cr
&= 
\limsup_{n\ra\infty}\frac1{r_n}\log\sup_{i\in\{1,\ldots m\}}
\mu_n (B_{x_i}^{\delta,y_i(x_i)}).
\end{aligned}
\end{equation*}

Now,
\begin{equation*}
\begin{aligned}
\label{}
\mu_n (B_{x_i}^{\delta,y_i(x_i)})
&=
\int^{}_{B_{x_i}^{\delta,y_i(x_i)}}\mu_n(z) 
\cr
&=
\int^{}_{B_{x_i}^{\delta,y_i(x_i)}}
e^{r_n\langle z,y_i(x_i) \rangle} 
e^{-r_n\langle z,y_i(x_i) \rangle}
\mu_n 
\cr
&\le
e^{\delta-r_n\langle x_i,y_i(x_i) \rangle}
\int^{}_{B_{x_i}^{\delta,y_i(x_i)}}
e^{r_n\langle z,y_i(x_i) \rangle }
\mu_n 
\cr
&\le
e^{\delta-r_n\langle x_i,y_i(x_i) \rangle}
\int^{}_{\calX}
e^{r_n\langle z,y_i(x_i) \rangle }
\mu_n
\end{aligned}
\end{equation*}
(integration in the $z$ variable).
Taking log and the limit,
\begin{equation*}
\begin{aligned}
\label{}
\limsup\frac1{r_n}\log\mu_n (B_{x_i}^{\delta,y_i(x_i)})
&\le -\langle x_i,y_i(x_i) \rangle+\delta+p(y_i(x_i))
\le -\pstar(x_i)+2\delta\le -\inf_C\pstar+2\delta.
\end{aligned}
\end{equation*}
Letting $\delta$ tend to zero completes the argument.

\subsection{Exponential tightness}
\label{ExpTightSubSec}

\begin{definition}
\label{}
A sequence of probability measures $\{\mu_n\} \subset P(\calX)$ is exponentially tight with normalization $r_n$ is for each $b\in(0, \infty 
) $ there exists a compact set $K_b\subset  \hbox{$\calX $ } $ such that 
$$
\limsup\frac1{r_n}\log \mu_n ( \hbox{$\calX $ } \setminus K_b)\le - b.
$$
\end{definition}

\begin{remark} {\rm
\label{}
The point of course is that $K_b$ is independent of $n $.
} \end{remark}

\begin{remark} {\rm
\label{}
Note, of course, that exponential tightness is automatic if $\calX$ is compact (or if there is a compact set containing the support of all the $\mu_n$)!
In particular, note that $\calX=P(B)$ (for some compact (finite-dimensional) 
manifold $B$) is a compact set by Prokhorov's theorem
\cite[p. 43]{Villani:oldandnew},\cite[Theorem 1.3]{AmbGig}: Let \hbox{$\calX $ } be a Polish space and
$P\subset P(\calX)$; then $P$ is pre-compact for the week topology if and only if for every $\epsilon>0$ there is a compact set $K_\eps\subset \hbox{$\calX $ } $ such that $\mu( \hbox{$\calX $ } \setminus K_\eps) \le \epsilon$ for all $\mu\in P$.
} \end{remark}

\begin{lemma}
\label{}
Suppose that the large deviation upper bound inequality holds
for $(\mu_n,r_n)$ for all compact sets. Suppose also that sequence of probability measures $\{\mu_n\} \subset P(\calX)$ is exponentially tight with normalization $r_n$. Then the large deviation upper bound inequality holds.
\end{lemma}

\begin{proof}
Then equality hold for all compact sets by assumption. So consider a close set $F$ that is not necessarily compact. Of course, $\mu_n (F) \le \mu_n (F\cap K_b) + \mu_n ( \hbox{$\calX $ } \setminus K_b) $. This is a very coarse inequality (since 
$\calX \setminus K_b$ is a large set!), but it does the job since 
$\mu_n(  \hbox{$\calX $ } \setminus K_b)$ is uniformly small and $F\cap K_b$ is compact (so we can apply to large deviation upper bound to it):
\begin{equation*}
\begin{aligned}
\label{}
\limsup\frac1{r_n}\log \mu_n (F)
&\le
\limsup\frac1{r_n}\log [\mu_n (F\cap K_b) + \mu_n ( \hbox{$\calX $ } \setminus K_b)] 
\cr
&\le
\max\{-b,\limsup\frac1{r_n}\log \mu_n (F\cap K_b)\}
\cr
&\le
\max\{-b,-\inf_{F\cap K_b} I\}
\cr
&\le
\max\{-b,-\inf_{F} I\},
\cr
\end{aligned}
\end{equation*}
and by choosing $b>\inf_F I$ (recall $F$ is closed and $I$ lsc and so the infimum is attained and $I>- \infty 
$ by assumption) we get
$\limsup\frac1{r_n}\log \mu_n (F)
\le
-\inf_{F} I
$.
\end{proof}

Thus, to complete the proof of the upper bound it remains to show:
\begin{lemma}
\label{exptighlemma}
The 
sequence of probability measures $\{\mu_n\} \subset P(\calX)$ is exponentially tight with normalization $r_n$.
\end{lemma}

\begin{proof}
For this proof we assume $\calX=\RR^n$ (otherwise, one needs to incorporate the exponential tightness assumption into the assumptions of Theorem \ref{GEThm}). 
This follows directly from the assumption that a moment generating function exists. Indeed, choose a coordinate $x_i$ and bound the tail in that direction:
$$
\begin{aligned}
\mu_n\{x_i\ge b\}
&=
\int_{\{x_i\ge b\}}
e^{-r_n\langle \theta,x \rangle}
e^{r_n\langle \theta,x \rangle} \mu_n(x)
\cr
&\le
e^{-r_nb|\theta|}
\int_{\{x_i\ge b\}}
e^{r_n\langle \theta,x \rangle} \mu_n(x)
\cr
&\le
e^{-r_nb|\theta|}
\int_{\calX}
e^{r_n\langle \theta,x \rangle} \mu_n(x)
\cr
&=
e^{-r_nb|\theta|}
p_n(\theta)
\cr\end{aligned}
$$
Now,
$$
\frac1{r_n}\log p_n(\theta)=p(\theta)+o(1), 
$$
so
$$
p_n(\theta)= e^{r_n(p(\theta)+o(1))}.
$$
So, fixing $\theta$ and then choosing $b>0$ sufficiently large ($p(\theta)$ is finite!),
and summing over all coordinate directions concludes the proof.
\end{proof}

\subsection{The lower bound }

Our goal is to show that $$
\liminf_{n\ra\infty}\frac1{r_n} \log \mu_n (O)\ge -\inf_O  \pstar , \qq \forall\; O \h{\ \ open in \ } \calX.
$$
Fix an open set $O $ and a point $z\in O $ where $\inf_O  \pstar$ is attained
up to some $\eps$. First, as in the proof of the upper bound, we will show that
$\frac1{r_n} \log \mu_n (O)$ is essentially equal to $\frac1{r_n} \log \mu_n (B_\delta)$
for some ball $B_\delta$ containing $z $.
Indeed, for any $B$ and any $\sigma\in\calX^*$,
\begin{equation}
\begin{aligned}
\label{munBfirsteq}
\frac1{r_n}\log\mu_n (B)
&= 
\frac1{r_n}\log \int_{B}
e^{-r_n\langle z,y_i \rangle}
e^{r_n\langle z,y_i \rangle} 
\mu_n (z)
\end{aligned}
\end{equation}
We want to do essentially the same computation as for the upper bound, 
except that now of course the inequality
$$
\int^{}_{B}
e^{r_n\langle z,y_i \rangle }
\mu_n 
\le \int^{}_{\calX }
e^{r_n\langle z,y_i \rangle }
\mu_n 
$$
goes in the wrong direction. To remedy that, we need to identify some way of localizing the integral so that, at least asymptotically, the integrals are equal.
The key is to notice that the relevant point for localizing is
$$
\sigma: = (\nabla p)^ {-1} (z).
$$
Consider
$$
e^{\langle \sigma,\,\cdot\, \rangle}\mu_n,
$$
or, rather, the associated probability measures
$$
\nu_{\sigma,n}: =
e^{\langle \sigma,\,\cdot\, \rangle}\mu_n\bigg/\int_{\calX}
e^{\langle \sigma,y \rangle}\mu_n(y).
$$

\begin{lemma}
\label{ldpnulemma}
The sequence of probability measures $\{\nu_{\sigma,n}\}$ localizes 
around $z=\nabla p(\sigma) $. More precisely, we have the upper bound
large deviation inequality for $\{(\nu_{\sigma,n}, r_n)\}$
with rate function $ \pstar - \langle \sigma,\,\cdot\, \rangle+ p (\sigma).$ 
\end{lemma}
\begin{remark} {\rm
\label{}
Recall that 
\begin{equation}
\begin{aligned}
\label{ppstar2eq}
 \pstar (z) =\langle \sigma, z\rangle - p (\sigma),
\end{aligned}
\end{equation}
and
\begin{equation}
\begin{aligned}
\label{ppstareq}
 \pstar (\,\cdot\,) +  p (\sigma) > \langle \sigma, \,\cdot\,\rangle
\qq \h{away from $z$}.
\end{aligned}
\end{equation}
We postpone the proof of Lemma  \ref{ldpnulemma}  to the end of the section.

Thus, the rate function in the statement is nonnegative with a unique zero at $z $. Thus the desired localization:
} \end{remark}
\begin{corollary}
\label{nuncor}
For any $\delta> 0 $,
$$
\lim_n  \nu_{\sigma,n}( B_\delta^z)=1.
$$
\end{corollary}
\begin{proof}
It suffices to show that
$$
\limsup_n\frac1{r_n}\log   \nu_{\sigma,n}(\calX\setminus  B_\delta^z)<0.
$$
By the large deviation  upper bound inequality for $\{(\nu_{\sigma,n}, r_n)\}$  and  \eqref{ppstareq} (note $\calX\setminus  B_\delta^z$ is closed!),
$$
\limsup_n\frac1{r_n}\log   \nu_{\sigma,n}(\calX\setminus  B_\delta^z)\le 
-\inf_{x\not\in B_\delta^z}[\pstar(x) - \langle \sigma,x \rangle+ p (\sigma)]
\le -C,
$$
for some $C=C (\delta) $, 
at desired.
\end{proof}
Thus, as we wished,
$$
\int^{}_{B_\delta^z}
\nu_{\sigma,n} 
=
\int^{}_{\calX}
\nu_{\sigma,n} +o(1)
=
1+o(1).
$$
Now we are ready to go back to \eqref{munBfirsteq},
\begin{equation*}
\begin{aligned}
\label{}
\frac1{r_n}\log\mu_n (B_\delta^z)
&= 
\frac1{r_n}\log 
\int_{B_\delta^z}
e^{-r_n\langle \sigma,x \rangle}
e^{r_n\langle \sigma,x \rangle} 
\mu_n (z)
\cr
&= 
\frac1{r_n}\log 
\int_{B_\delta^z}
e^{-r_n\langle \sigma,x \rangle}
p_n(\sigma)\nu_{\sigma,n}(x)
\cr
&= 
p(\sigma)+o(1)+\frac1{r_n}\log 
\int_{B_\delta^z}
e^{-r_n\langle \sigma,x \rangle}
\nu_{\sigma,n}(x)
\cr
&\ge 
p(\sigma)+o(1)+\frac1{r_n}\log 
\inf_{B_\delta^z}
e^{-r_n\langle \sigma,x \rangle}
+\frac1{r_n}\log \int_{B_\delta^z}
\nu_{\sigma,n}(x)
\cr
&\ge 
p(\sigma)+o(1)-\langle \sigma,z \rangle-\delta
+o(1)
\cr
&= 
- \pstar(z)+o(1)-\delta,
\cr
\end{aligned}
\end{equation*} 
where we used Corollary \ref{nuncor} and  \eqref{ppstar2eq}. 
Letting first $n$ go to infinity and then $\delta$ go to zero concludes the proof.

\begin{proof}
[Proof of  Lemma  \ref{ldpnulemma}]
First, observe that the Legendre transform of
$$
p (\,\cdot\, +\sigma) - p (\sigma)
$$
is
\begin{equation}
\begin{aligned}
\label{pstar3eq}
 \pstar - \langle \sigma,\,\cdot\, \rangle+ p (\sigma).
\end{aligned}
\end{equation}
Thus, it suffices to show that the moment generating function of $\{\nu_{\sigma,n}\}$
is
$
p (\,\cdot\, +\sigma) - p (\sigma)
$.
Indeed,
\begin{equation*}
\begin{aligned}
\label{}
\lim \frac1{r_n}\log \int^{}_{ \calX } e^{r_n\langle \theta , x \rangle }
\nu_{\sigma,n} (x)
&=
\lim \frac1{r_n}\log \int^{}_{ \calX } e^{r_n\langle \theta , x \rangle }
e^{\langle \sigma,x \rangle}\mu_n(x)\bigg/p_n(\sigma)
\cr&=
-p(\sigma)\lim \frac1{r_n}\log \int^{}_{ \calX } e^{r_n\langle \theta+\sigma , x \rangle }\mu_n(x)
\cr
&=
-p(\sigma)+p(\theta+\sigma).
\end{aligned}
\end{equation*}
All the assumptions of  Theorem \ref{GEThm}
 are met for this generating function since they are met for $p$. 
Under those assumptions we have already established the upper bound inequality in
Theorem \ref{GEThm}. Therefore, 
we have the upper bound
large deviation inequality for $\{(\nu_{\sigma,n}, r_n)\}$
with rate function  \eqref{pstar3eq}. 
 \end{proof}

\section{LDP without moment generating functions}
\label{LDPwithoutSec}

Sometimes, a large deviation principle holds even when the assumptions of Theorem \ref{GEThm} are not satisfied. For one, a moment generating function may not exist. Also, the rate function can sometimes be nonconvex (not that Theorem \ref{GEThm} guarantees the rate function will be convex, being the Legendre transform of the moment generating function).  The following result nevertheless characterizes large deviation principles and gives a useful tool to show their existence.
It is sort of a converse for Lemma \ref{ILemma}. 

\begin{prop}
\label{ldpprop}
Let \hbox{$\calX $ } be a compact metric space. LDP$\,(\mu_n , r_n)$ if and only if
\begin{equation}
\begin{aligned}
\label{grateeq}
\lim _{d\ra 0}\limsup_{n\ra  \infty } \frac1{r_n}\log \mu_n (B_d (x))=
\lim _{d\ra 0}\liminf_{n\ra  \infty } \frac1{r_n}\log \mu_n (B_d (x)),
\q \forall x\in  \hbox{$\calX $}.
\end{aligned}
\end{equation}
The rate function is then equal to minus \eqref{grateeq}.
\end{prop}

\begin{proof}
\noindent
$\ul{\Longleftarrow}$:\
Suppose \eqref{grateeq} holds, and denote either side by $-g(x)$. 
Note that $g$ is indeed a rate function: it is evidently not negative, and it is lower semicontinuous because the super level sets $\{g> a\}$ are open as if 
$g(x)>a$, i.e., $-g(x)<-a$, then because the limits in \eqref{grateeq}
are decreasing in $d$, then for some 
$B_d(x)$,
\begin{equation*}
\begin{aligned}
\label{}
\liminf_{n\ra  \infty } \frac1{r_n}\log \mu_n (B_d (x))<-a,
\end{aligned}
\end{equation*} 
so for every $y\in B_d (x)$ choosing $d'$ so that $B_{d'}(y)\subset B_d (x)$,

\begin{equation*}
\begin{aligned}
\label{}
\liminf_{n\ra  \infty } \frac1{r_n}\log \mu_n (B_{d'} (y))\le
\liminf_{n\ra  \infty } \frac1{r_n}\log \mu_n (B_d (x))<-a,
\end{aligned}
\end{equation*} 
so $-g(y)<-a$.

To prove the lower bound, let $O\in\calX$ be open, fix $\eps>0$, and let $x\in O$
be such that $g(x)\le\inf_O g+\eps$, 
and let $d>0$ be such that $B_d(x)\subset O$. Then
\begin{equation*}
\begin{aligned}
\liminf_{n\ra  \infty } \frac1{r_n}\log \mu_n (O)
&\ge
\liminf_{n\ra  \infty } \frac1{r_n}\log \mu_n (B_\delta (x))
\cr
&\ge 
\lim _{d\ra 0}
\liminf_{n\ra  \infty } \frac1{r_n}\log \mu_n (B_\delta (x))
\cr
&=-g(x)=\ge-\inf_O g-\eps.
\end{aligned}
\end{equation*}
Now let $\epsilon $ tend to zero.

To prove the upper bound, let $C\in\calX$ be closed.
Because $\calX$ is compact, so is $C$.
Similar to the proof of \S\ref{uppbdsubsec} we cover $F$ with
finitely many balls and then 
\begin{equation*}
\begin{aligned}
\label{}
\limsup_{n\ra\infty}\frac1{r_n} \log \mu_n (C)
&\le 
\limsup_{n\ra\infty}\frac1{r_n} \log \sum_{i = 1}^m\mu_n (B_{d_i}(x_i))
\cr
&= 
\limsup_{n\ra\infty}\frac1{r_n} \log\big( 
m\sup_{i\in\{1,\ldots m\}}\mu_n (B_{d_i}(x_i))
\big)\cr
&= 
\limsup_{n\ra\infty}\bigg[\frac1{r_n} \log 
m + \frac1{r_n}\log\sup_{i\in\{1,\ldots m\}}\mu_n (B_{d_i}(x_i))
\bigg]\cr
&= 
\limsup_{n\ra\infty}\frac1{r_n}\log\sup_{i\in\{1,\ldots m\}}\mu_n (B_{d_i}(x_i))
\cr
&= 
\max_{i\in\{1,\ldots m\}}\limsup_{n\ra\infty}\frac1{r_n}\log\mu_n (B_{d_i}(x_i))
\cr
&= 
\limsup_{n\ra\infty}\frac1{r_n}\log\mu_n (B_{d_1}(x_1)),
\end{aligned}
\end{equation*}
where the last equality is without loss of generality.
Fix $\eps>0$. 
By \eqref{grateeq}, for each $x\in  \hbox{$\calX $ } $ there exists $d(x,\eps)>0$  
such that 
$$
\limsup_{n\ra\infty}\frac1{r_n}\log\mu_n (B_{d(x,\eps)}(x))\le \max\{-g(x)+\eps,-1/\eps\}
$$
(note that if $g $  is finite one can simplify the right-hand side to $-g(x)+\eps$).
Now, cover $C$ with the balls $\cup_{x\in F}B_{d(x,\eps)}(x)$; then, by compactness of $C$ we may choose a finite sub-cover, and thus, we may assume that we have chosen
$d_1=d(x_1,\eps)$!
Thus,
\begin{equation*}
\begin{aligned}
\label{}
\limsup_{n\ra\infty}\frac1{r_n} \log \mu_n (C)
&\le  
\limsup_{n\ra\infty}\frac1{r_n}\log\mu_n (B_{d_1}(x_1))
\cr
&\le 
\max\{-g(x_1)+\eps,-1/\eps\},
\end{aligned}
\end{equation*}
and letting $\epsilon $ tend to zero,
\begin{equation*}
\begin{aligned}
\label{}
\limsup_{n\ra\infty}\frac1{r_n} \log \mu_n (C)
&\le 
-g(x_1)
\le 
-\inf_F g.
\end{aligned}
\end{equation*}

\noindent
$\ul{\Longrightarrow}$:\
Conversely, suppose LDP$\,(\mu_n , r_n)$ with rate function $I$.
By the large deviation lower bound,
$$
I(x)\ge \inf_O I\ge
-\liminf_{n\ra  \infty } \frac1{r_n}\log \mu_n (O),
$$
for any open set $O$ containing $x $, so putting $O=B_d(x)$ and supping
over $d$,
$$
I(x)\ge
-\lim _{d\ra 0}
\liminf_{n\ra  \infty } \frac1{r_n}\log \mu_n (B_d (x)).
$$
By the large deviation upper bound,
$$
-\limsup_{n\ra  \infty } \frac1{r_n}\log \mu_n (\overline{B_d (x)})
\ge
\inf_{\overline{B_d (x)}} I
,
$$
so 
$$
I(x)\ge
-\lim _{d\ra 0}
\liminf_{n\ra  \infty } \frac1{r_n}\log \mu_n (B_d (x))
\ge
-\lim _{d\ra 0}
\limsup_{n\ra  \infty } \frac1{r_n}\log \mu_n (\overline{B_d (x)})
\ge
\lim _{d\ra 0}\inf_{\overline{B_d (x)}} I
,
$$
and
it suffices to show that
$$
\lim_{d\ra 0}\inf_{\overline{B_d (x)}}I\ge I(x)
.$$
If not, there exists $x_k$ such that
$
\lim_k I(x_k)< I(x)
$
and $x_k\ra x$, contradicting lower semicontinuity of $I$.
\end{proof}

\section{Optimal transport}
\lb{OTSec}

The problem of optimally transporting a given probability measure $\mu\in P(X)$ 
(the source measure)
to another given probability measure $\nu\in P(Y)$ (the target  measure) has a long history, going back to Monge in the 18th century. 
It is the problem of finding a measurable  map $T:X\ra Y$ satisfying
\begin{equation}
\begin{aligned}
\label{Teq}
T_\# \mu=\nu
\end{aligned}
\end{equation}
and minimizing
$$
\int_X c(x,T(x))\mu(x),
$$
where $c:X\times Y\ra \RR$ is some given {\it cost function}, typically
$ c(x,y)=|x-y|^2$. This integral is the total cost, and $c(x,T(x))\mu(dx)$ is the infinitesimal cost of transporting $x$ to $T(x)$, with $\mu(dx)$ measuring the amount of mass at the source point $x$. By abuse of notation we denote the latter by $\mu(x)$ and not $d\mu(x)$ or  $\mu(dx)$.

As in previous sections, one should think of
$X=Y=\RR^n$ or 
$X=Y=P(\RR^d)$ as the typical examples in our course
for the underlying spaces. Typical examples for the measures to be transported  include uniform measures
 $1_\Omega$ (for a unit-volume set $\Omega $) and the empirical measures
\begin{equation*}
\begin{aligned}
\label{}
\delta^n (x_1,\ldots , x_n): =\frac1{n}\sum_{i = 1}^n\delta_{x_i}.
\end{aligned}
\end{equation*}
Choosing both the source and the target measures to be empirical measures with the same number of point masses (i.e., choosing $\mu$ and $\nu$ in the image of $\delta^n$
for the same $n $) gives rise to the so-called {\it discrete optimal transport problem. } The solution is then given by a permutation $\sigma\in S_n $ on the set of $n$ letters, satisfying
\begin{equation}
\begin{aligned}
\label{}
\sum_{i = 1}^nc(x_i,y_i)\le \sum_{i = 1}^nc(x_i,y_{\sigma(i)}).
\end{aligned}
\end{equation}
In the prototypical case of squared distance cost some cancellations give that
\begin{equation*}
\begin{aligned}
\label{}
\sum_{i = 1}^n|x_i-y_i|^2\le \sum_{i = 1}^n|x_i-y_{\sigma(i)}|^2
\end{aligned}
\end{equation*}
can be rewritten as
\begin{equation}
\begin{aligned}
\label{optdisceq}
\sum_{i = 1}^n-\langle x_i,y_i\rangle \le 
\sum_{i = 1}^n-\langle x_i,y_{\sigma(i)}\rangle .
\end{aligned}
\end{equation}
So, the cost $|x-y|^2$ is really `equivalent' to the cost $-\langle x,y \rangle$. More generally, if $c(x,y)=d(x,y)+f(x)+g(y)$ then $c$ and $d$ are equivalent: 
$$
\int_X c(x,T(x))\mu(x)=
\int_X d(x,T(x))\mu(x)+ \int_X f\mu+ \int_X g(T(x))\mu=
\int_X d(x,T(x))\mu(x)+ \int_X f\mu+ \int_Y g\nu
$$
(as $\int_X f\mu+ \int_Y g\nu$ is a constant completely determined by the ``data" $\mu,\nu$),
where in the last equation we used that
$$
\int_X g\circ T\mu=\int_Y g\nu,
$$
by  \eqref{Teq} \cite[(9)]{Villani}.
 
More generally, one can search for an optimal transportation {\it plan}. 
Given a product space, say $X\times Y$, equipped with projection maps $\pi_X:X\times Y\ra X$ and $\pi_Y:X\times Y\ra Y$, the marginals of
a measure $\gamma\in P(X\times Y)$ are $(\pi_X)_\#\gamma$ and $(\pi_Y)_\#\gamma$.

\begin{definition}
\label{optdef}
A transportation plan is a probability measure $\gamma\in P(X\times Y)$ whose marginals are $\mu$ and $\nu$. We denote this by $\gamma\in \Pi(\mu,\nu)$.
\end{definition}

I.e., $\gamma(A\times Y)=\mu(A)$ and $\gamma(X\times B)=\nu(B)$
for all Borel $A\subset X$ and $B\subset Y$

\begin{definition}
\label{}
An {\it optimal transportation plan } is a transportation plan minimizing
$$
\int^{}_{X\times Y}c\gamma. 
$$
\end{definition}

The ``best" transport plan is the one coming from transport map $T:X\ra Y$. 
Denote by $\id\otimes T:X\ra X \times Y$ the map $x\mapsto (x,T(x))$. 
Indeed,  $\gamma:=(\id\otimes T)_\# \mu\in \Pi(\mu,\nu)$ 
since $(\id\otimes T)_\# \mu(A\times Y)=\mu\big((\id\otimes T)^{-1}(A\times Y)\big)=
\mu(A)$ and $(\id\otimes T)_\# \mu(X\times B)=\mu\big((\id\otimes T)^{-1}(X\times B)\big)=\mu(T^{-1}(B))=\nu(B)$ since $T_\# \mu=\nu$.
Our goal will be to show that under some natural assumptions the  {\it optimal } plan must be of such a form, i.e., supported on the graph of a map.

For example, in the case of empirical measures, a transportation plan must be coming from a map represented by a permutation
(in other words, it must be supported on the graph of a permutation in the product space):
if one of the source points is not in $\supp\gamma$ then the first marginal condition is violated ($\gamma(\{x_i\}\times Y)=0$ while $\mu\{x_i\}=1/n$
but $\gamma(A\times Y)=\mu(A)$), while if one of the target points is not in $\supp\gamma$ then the second marginal condition is violated.
Thus $\sigma\in S_n$, or more precisely,
$$
\gamma:=\frac 1n \sum \delta_{(x_i,y_{\sigma(i)})}=\delta_{\gr(T)},
$$
(where $T:x_i\mapsto y_{\sigma(i)}$)
 is optimal if and only if \eqref{optdisceq} holds.

\subsection{From the discrete problem to the general one}
\label{}
A beautiful part of the story is that the discrete problem actually is the key to understanding the general transport problem. Equation  
\eqref{optdisceq} leads to the following definition (we replace
$n$ in \eqref{optdisceq}  with $m$ for the following discussion).

\begin{definition}
\label{}
A set $A\subset X\times Y $ is cyclically monotone if
\eqref{optdisceq}  holds for any $\{(x_i,y_i)\}_{i=1}^m\subset A$, $m\in \NN$,
and any $\sigma\in S_m$.
\end{definition}

Cyclical monotonicity essentially characterizes convexity. More precisely, the graph of the sub differential of a convex function $f$ is cyclically monotone: if $y_i\in\del f(x_i)$ 
$$
f(z)\ge f(x_i)+\langle z-x_i,y \rangle, \q \forall z, 
$$
so taking $z=x_{i+1}$ (with $x_{m+1}=x_1$) and adding up the equations yields 
\eqref{optdisceq}.
Conversely, to a cyclically monotone set $A$ we can associate a convex function $f_A$ such that $A\subset\gr(\del f_A)$ as follows. Fix $(x_0,y_0)\in A$ and set
$$
f_A(x):=\sup_{m\in\NN}\sup_{\{(x_i,y_i)\}_{i=1}^m\subset A}
\Big\{
\langle x-x_m,y_m \rangle +
\langle x_m-x_{m-1},y_{m-1} \rangle +
\ldots +
\langle x_1-x_{0},y_{0} \rangle
\Big\}.
$$
Note that $f_A$ is not $\pm \infty $ :
\begin{claim}
\label{fAclaim}
$f_A(x_0)=0$.
\end{claim}

\begin{remark} {\rm
\label{}
We will eventually prove much more, namely, that $f_A$ is nowhere $\pm \infty $.
} \end{remark}

\begin{proof}
First, $f_A(x_0)\le 0$  from
\eqref{optdisceq} with $m$ replaced by $m +1 $
and $\sigma(i)=i+1$. Second, $f_A(x_0)\ge 0$ by
putting $m=1$ and $(x_1,y_1)=(x_0,y_0)$ in the definition of $f_A$.
\end{proof}

Finally, if $(a,b)\in A$, want to show $b\in \del f_A(a)$
(the {\it sub-differential} of $f_A$), i.e., 
\begin{equation}
\begin{aligned}
\label{fAeq}
f_A(z)\ge f_A(a)+\langle z-a,b \rangle,\q \forall z. 
\end{aligned}
\end{equation}
Given any $\epsilon>0$ there is some $m\in\NN$ 
and some $\{(x_i,y_i)\}_{i=1}^m\subset A$ such that
\begin{equation}
\begin{aligned}
\label{epscycl}
f_A(a)-\eps=
\langle a-x_m,y_m \rangle +
\langle x_m-x_{m-1},y_{m-1} \rangle +
\ldots +
\langle x_1-x_{0},y_{0} \rangle.
\end{aligned}
\end{equation}
Thus,
$$
f_A(a)
+\langle z-a,b \rangle
=
\eps+
\langle z-a,b \rangle+
\langle a-x_m,y_m \rangle +
\langle x_m-x_{m-1},y_{m-1} \rangle +
\ldots +
\langle x_1-x_{0},y_{0} \rangle
\le
\eps+f_A(z),
$$
putting $m+1$ and $\{(x_i,y_i)\}_{i=1}^m\cup\{(a,b)\}$ in the definition of $f_A$.
Letting $\eps\ra0$ concludes the proof of  \eqref{fAeq}.

\begin{exer} {\rm
\label{}
Find the mistake in the previous argument.
} \end{exer}

\noindent
{\bf Solution.} 
The problem was that we were implicitly assuming that
$
f_A(a)< \infty.
$
Indeed, if $
f_A(a)= \infty$ one cannot, given any $\eps>0$, 
find $\{(x_i,y_i)\}_{i=1}^m\subset A$ such that 
 \eqref{epscycl} holds.
Instead, let $t\in\RR$ be any number satisfying $t<f_A(a)$
(possibly, $f_A(a)=\infty$). Now, there do exist, by definition,
$\{(x_i,y_i)\}_{i=1}^m\subset A$ such that
\begin{equation}
\begin{aligned}
\label{tcycl}
t<
\langle a-x_m,y_m \rangle +
\langle x_m-x_{m-1},y_{m-1} \rangle +
\ldots +
\langle x_1-x_{0},y_{0} \rangle.
\end{aligned}
\end{equation}
Thus, for all $z$,
\begin{equation}
\begin{aligned}
\label{tmeq}
t
+\langle z-a,b \rangle
<
\langle z-a,b \rangle+
\langle a-x_m,y_m \rangle +
\langle x_m-x_{m-1},y_{m-1} \rangle +
\ldots +
\langle x_1-x_{0},y_{0} \rangle
\le
f_A(z),
\end{aligned}
\end{equation}
by putting $m+1$ and $\{(x_i,y_i)\}_{i=1}^m\cup\{(a,b)\}$ in the definition of $f_A$. 
Supping over all  $t$ in  \eqref{tmeq},
$$
f_A(a)+\langle z-a,b \rangle=\sup_{t<f_A(a)}t+\langle z-a,b \rangle
\le f_A(z),
$$
as desired, i.e., $b\in\del f_A(a)$, unless $f_A(a)=\infty$
(in which case $\del f_A(a)=\emptyset$ by definition).
To exclude this, i.e.,
 to show $f_A$ is always finite, we put $z =0$ in
\eqref{tmeq}, and use Claim \ref{fAclaim}, 
$$
t
+\langle 0-a,b \rangle
<
f_A(0)=0,
$$
so we get the a priori estimate
$$
t< \langle a,b \rangle, \qq \h{for any $t<f_A(a)$}.
$$
Hence, $f_A(a)=\sup_{t<f_A(a)}t\le \langle a,b \rangle$ and in particular
$f_A(a)$ is finite (obviously $f_A>-\infty$ since it is a supremum
of finite quantities over a nonempty set). 

Recall the definition of the sub-differential $\partial f$ 
\eqref{fAeq} of a convex function $f:\RR^n\ra \RR$. The graph of $\partial f$
is defined as $\gr(\del f):=\{(x,y)\,:\, x\in\RR^n, y\in \partial f(x)\}.$
Thus, we have shown:

\begin{theorem}
\label{RockThm}
{\rm (Rockafellar's Theorem)}
A set $A\subset \RR^n\times \RR^n $ is cyclically monotone if
and only if $A\subset\gr(\del f)$ for a convex function $f:\RR^n\ra \RR$.
\end{theorem}

\begin{remark} {\rm
\label{}
Note, again, that part of the  conclusion of the theorem is
that $f$ is finite (which should not come as a surprise, as also the set $A$ is ``finite-valued" by assumption; the converse direction is obvious since if $f$ is finite-valued then so is $\del f$).
In fact, 
$$
f_A(x)\le \inf_{(x,y)\in A}\langle x,y\rangle
\le h_{\{x\}\times\RR^n\cap A}(x),
$$ 
where $h_K$ is the support function of the set $K$ with
equality in the last inequality if $A$ is graphical (i.e.,
of the form $\gr(\del f)$ with $f$ as above).
} \end{remark}

The Fundamental Theorem of OT gives the final chain connecting the discrete problem to the continuous problem: the support of an optimal transport plan is cyclically monotone. Thus, by a previous observation it is contained in the graph of the sub differential of a convex function. Since a convex function is differentiable away from a (Lesbegue) measure zero set as measures 
$(\id\otimes \del f)_\# \mu= (\id\otimes \nabla f)_\# \mu$ 
where on the right hand side by $\nabla f$ we mean (by some abuse of notation)
the gradient {\it map} restricted to those $x\in\RR^n$ where $\partial f(x)$
is a singleton (that we then denote by $\nabla f(x)$); 
note also that assuming $\mu$ is absolutely continuous, Lebesgue measure zero sets  are also $\mu$-measure zero sets. Thus, we have come full circle, and solved the original transportation problem in terms of a  {\it map. }

\begin{theorem} {\rm 
(Fundamental Theorem of Optimal Transport)
}
\label{ftotthm}
Let $\gamma\in \Pi(\mu,\nu)$.
Then, 

\noindent
$\gamma$ optimal\ \ \ $\Leftrightarrow$\ \ \ 
$\supp\gamma$ is cyclically monotone
\ \ \ $\Leftrightarrow$\ \ \
exists $f$ convex such that $\supp\gamma\subset \gr(\del f)$. 
\end{theorem}

\begin{proof}
By Theorem \ref{RockThm}, it suffices to show the first equivalence, but
we will actually only use the hard part of Theorem \ref{RockThm}
and proceed to prove the implications cyclically. 
First, suppose that for some
$\gamma\in \Pi(\mu,\nu)$, 
$\supp\gamma$ is cyclically monotone. By Theorem \ref{RockThm},
there exists $f$ convex such that $\supp\gamma\subset \gr(\del f)$. Thus,
for every $\tilde\gamma\in \Pi(\mu,\nu)$, 
\begin{equation}
\begin{aligned}
\label{ftoteq}
\int_{X\times Y} -\langle x,y\rangle \gamma
&=
\int_{\supp\gamma} -\langle x,y\rangle \gamma
\cr
&=
\int_{\supp\gamma} [-f(x)-f^\star(y)] \gamma
\cr
&=
-\int_X f\mu - \int_Y f^\star \nu.
\cr
&=
\int_{X\times Y}
[-f(x)-f^\star(y)] \tilde\gamma
\cr
&\le
\int_{X\times Y} -\langle x,y\rangle \tilde\gamma,
\end{aligned}
\end{equation}
so by definition $\gamma $  is optimal.

Next, assume that $\gamma $  is optimal. We want to show 
that $\supp\gamma$ is cyclically monotone.
Fix $\{(x_i,y_i)\}_{i=1}^m\subset \supp\gamma$.
To that end, we carefully construct $\tilde\gamma\in\Pi(\mu,\nu)$ of the form $\tilde\gamma:=\gamma +\eta$ with
\begin{equation}
\begin{aligned}
\label{approxcycleq}
0\le \int_{X\times Y}
 -\langle x,y\rangle\tilde\gamma
-
\int_{X\times Y}
 -\langle x,y\rangle\gamma
\approx
\sum_{i = 1}^n-\langle x_i,y_{\sigma(i)}\rangle
-
\sum_{i = 1}^n-\langle x_i,y_i\rangle.
\end{aligned}
\end{equation}

Of course, the idea is to construct the positive part of $\eta$ to be concentrated near 
$\{(x_i,y_{\sigma(i)})\}_{i=1}^m$
and the negative part of $\eta$ to be concentrated near
$\{(x_i,y_i)\}_{i=1}^m$. We have to do this in such a way that $\gamma+\eta$ is still admissible (i.e., a transport plan).
Equivalently, $(\pi_X)_{\#}\eta=0,\,(\pi_Y)_{\#}\eta=0$.

For the construction, we fix $\epsilon >0 $. Set
$$
\Gamma:=\prod_{i=1}^m 
\frac
{
\gamma|_{B_\eps(x_i)\times B_\eps(y_i)}
}
{
|\gamma\big(B_\eps(x_i)\times B_\eps(y_i)\big)|
}
$$
This is an auxiliary probability measure on $P\big((X\times Y)^m\big)$.
It is useful, because it's marginals allow us to cyclically modify the way $\gamma$ transports: in order to transport $B_\eps(x_i)$  to 
$B_\eps(y_{\sigma(i)})$ instead of to $B_\eps(y_i)$ we would add 
$$
(\pi_{B_\eps(x_i)}, \pi_{B_\eps(y_{\sigma(i)})})_{\#}\Gamma
-
(\pi_{B_\eps(x_i)}, \pi_{B_\eps(y_i)})_{\#}\Gamma. 
$$
to $\gamma$. So, overall, we set
\begin{equation}
\begin{aligned}
\label{etaeq}
\eta:=
\frac{\min_i |\gamma\big(B_\eps(x_i)\times B_\eps(y_i)\big)|}{m}
\sum_{i=1}^m
\Big[(\pi_{B_\eps(x_i)}, \pi_{B_\eps(y_{\sigma(i)})})_{\#}\Gamma
-
(\pi_{B_\eps(x_i)}, \pi_{B_\eps(y_i)})_{\#}\Gamma\Big]. 
\end{aligned}
\end{equation}
The constant ${\min_i |\gamma\big(B_\eps(x_i)\times B_\eps(y_i)\big)|}/{m}$ in front of some ensures that $\gamma+\eta$ is still a positive measure
(recalling that $|\gamma\big(B_\eps(x_i)\times B_\eps(y_i)\big)|$ appears in the denominator of $\Gamma$, so the largest negative term inside the
brackets in \eqref{etaeq} is ${\min_i |\gamma\big(B_\eps(x_i)\times B_\eps(y_i)\big)|}$ and there are at most $m$ of these negative terms).
Next, to show $(\pi_Y)_{\#}\eta=0$ amounts to
$
\eta(X\times B)=0
$
for each $B$, and indeed, up to a positive factor 
$
(\pi_Y)_{\#}\eta(B)=\eta(X\times B)
$ equals
\begin{equation*}
\begin{aligned}
\label{}
&\sum_{i=1}^m
\Big[(\pi_{B_\eps(x_i)}, \pi_{B_\eps(y_{\sigma(i)})})_{\#}\Gamma(X\times B)
-
(\pi_{B_\eps(x_i)}, \pi_{B_\eps(y_i)})_{\#}\Gamma(X\times B)\Big]
\cr
&= 
\sum_{i=1}^m
\Gamma(X\times Y\times\cdots\times
X\mskip-25mu\mathoverr{\h{\rmfoot $\mathfootnote\sigma(i)$-th slot}}
{\times }\mskip-25mu B\times\cdots\times X\times Y)
\cr
&\q -\sum_{i=1}^m
\Gamma(X\times Y\times\cdots\times
X\mskip-15mu\mathoverr{\h{\rmfoot $\mathfootnote i$-th slot}}{\times }\mskip-15mu B\times\cdots\times X\times Y)
=
0
\end{aligned}
\end{equation*}
Finally, $\,(\pi_X)_{\#}\eta=0$ is easier as 
\begin{equation*}
\begin{aligned}
\label{}
&\sum_{i=1}^m
\Big[(\pi_{B_\eps(x_i)}, \pi_{B_\eps(y_{\sigma(i)})})_{\#}\Gamma(A\times Y)
-
(\pi_{B_\eps(x_i)}, \pi_{B_\eps(y_i)})_{\#}\Gamma(A\times Y)\Big]
\cr
&= 
\sum_{i=1}^m
\Gamma(X\times Y\times\cdots\times
A\mskip-25mu\mathoverr{\h{\rmfoot $\mathfootnote\sigma(i)$-th slot}}
{\times }\mskip-25mu Y\times\cdots\times X\times Y)
\cr
&\q -\sum_{i=1}^m
\Gamma(X\times Y\times\cdots\times
A\mskip-15mu\mathoverr{\h{\rmfoot $\mathfootnote i$-th slot}}{\times }\mskip-15mu Y\times\cdots\times X\times Y)
=
\sum_{i=1}^m 0=0
\end{aligned}
\end{equation*}
(i.e., is term-by-term zero).

Thus, we have shown  \eqref{approxcycleq}
 up to $o(\eps)$. Letting $\eps\ra 0$ proves  \eqref{optdisceq},
as claimed.
\end{proof}

\subsection{Dual formulation}
\label{}

A rather immediate consequence of Theorems \ref{RockThm} and \ref{ftotthm} is the following dual formulation of the optimal transportation problem in terms of an optimization problem on {\it functions } instead of measures.

\begin{theorem} 
\label{dualthm}
Let $c(x,y)=-\langle x,y \rangle$.
\begin{equation}
\label{DualFormEq}
\inf_{\gamma\in\Pi(\mu,\nu)}
\int^{}_{X\times Y}c\gamma
=
\sup_{f(x)+g(y)\le c(x,y)}
\Big[
\int_Xf\mu+\int_Yg\nu
\Big].
\end{equation}
\end{theorem}

\begin{proof}
According to Lemma \ref{optexistlemma} there exists $\gamma$ realizing the infimum on the left-hand side.
Let $f,g$ be such that 
$f(x)+g(y)\le -\langle x,y \rangle$. Then
$$
\int^{}_{X\times Y}c\gamma
\ge
\int^{}_{X\times Y}
(f(x)+g(y))\gamma
=\int_Xf\mu+\int_Yg\nu.
$$
 It thus remains to show 
\begin{equation}
\begin{aligned}
\label{phidualeq}
\int^{}_{X\times Y}c\gamma
\le
\sup_{f(x)+g(y)\le c(x,y)}
\Big[
\int_Xf\mu+\int_Yg\nu
\Big].
\end{aligned}
\end{equation}
Theorems \ref{RockThm} and \ref{ftotthm} imply that $\supp\,\gamma\subset\gr(\nabla\phi)$
for some convex function $\phi$. Since $\phi(x)+\phi^\star(y)\ge -c(x,y)$
with the quality if and only if $(x,y)\in\gr(\del \phi)$,
\begin{equation*}
\begin{aligned}
\label{}
\int^{}_{X\times Y}c\gamma
&=
\int^{}_{\supp\gamma}c\gamma
\cr
&=
\int^{}_{\gr(\del \phi)}c\gamma
\cr&=
\int^{}_{X\times Y}[-\phi(x)-\phi^\star(y)]\gamma
\cr&=
\int^{}_{X\times Y}[
-\phi(x)-\phi^\star (y)]\gamma
\cr&=
-
\int_X\phi\mu-\int_Y\phi^\star\nu.
\end{aligned}
\end{equation*}
Thus,  \eqref{phidualeq} holds as already the pair $(f,g)=
(-\phi,-\phi^\star)$ equals the left-hand side.
\end{proof}

In fact, we saw that the dual formulation can be given in 
terms of a single (and convex) function. Also, we
play a bit with the signs, to get:

\begin{cor}{\rm (Dual formulation of optimal transportation)}
\label{dualCor}
$$
\sup_{\gamma\in\Pi(\mu,\nu)}
\int^{}_{X\times Y}\langle x,y \rangle\gamma
=
\inf_{f\in C(X)}
\Big[
\int_Xf\mu+\int_Y f^\star\nu
\Big]
=
\inf_{f\in \Cvx(X)}
\Big[
\int_Xf\mu+\int_Y f^\star\nu
\Big].
$$
\end{cor}

\begin{lemma}
\label{optexistlemma}
The infimum on the left hand side of \eqref{DualFormEq} is attained.
\end{lemma}

\begin{proof}
The proof follows Ambrosio--Gigli who work in a more general setting 
\cite[\S1.1]{AmbGig}.
Since $c: X\times Y\ra \RR$ is continuous (in fact, lower semicontinuous is enough, with slightly more work
\cite[Theorem 1.2]{AmbGig}) then $\gamma\mapsto \int^{}_{} c\gamma$ is continuous with respect to the weak topology. Since 
\begin{equation}
\begin{aligned}
\label{ulameq}
\gamma (X\times Y\setminus K_1\times K_2)\le \mu (X\setminus K_1) +\mu (Y\setminus K_2), 
\end{aligned}
\end{equation}
for {\it any } $\gamma\in\Pi(\mu,\nu)$, it follows that  $\Pi(\mu,\nu)$ satisfies the assumptions of Prokhorov's Theorem (see \cite[Theorem 1.3]{AmbGig}): indeed, the right-hand side of \eqref{ulameq} can be made arbitrarily small by Ulam's Theorem (any Borel probability measure
on a Polish space is concentrated on a compact set up to an arbitrarily small error)
applied to the Polish measure spaces $(X,\mu)$  and $(Y,\nu)$. Thus, $\Pi(\mu,\nu)$
is pre-compact. The closure of a pre-compact set is by definition compact, it suffices to show that $\Pi(\mu,\nu)$ is actually closed (with respect to the weak topology), but this is immediate since $\int f\mu=\int f(x)\gamma_n\ra
\int f(x)\gamma=\int f\mu$ (note that $(x,y)\mapsto f(x)$ for $f\in C(X)$ is in $C(X\times Y) $ so then $\alpha\mapsto \int f(x)\alpha$ is continuous
with respect to the weak topology)  and similarly for the other marginal.
Since a lower semicontinuous functional attains its infimum on a compact set, we are done.
\end{proof}

\subsection{The Legendre transform of Wasserstein distance}
\label{}
Denote by
\begin{equation}
\begin{aligned}
\label{CXeq}
C(X):= C^0(X)\cap L^\infty(X)
\end{aligned}
\end{equation}
the continuous and bounded functions on $X$.

Let $J_\nu:C(X)\ra \RR$ be
\begin{equation}
\begin{aligned}
\label{jeq}
J_\nu(f):=\int^{}_{Y}f^\star \nu. 
\end{aligned}
\end{equation}
Denote by $W_2^2:P(X)\times P(Y)\ra\RR$
\begin{equation}
\begin{aligned}
\label{W22eq}
W_2^2(\mu,\nu):=
\inf_{\gamma\in\Pi(\mu,\nu)}
\int^{}_{X\times Y}c\gamma
\end{aligned}
\end{equation}
the {\it Wasserstein distance between $\mu$ and $\nu$}. 
It will be convenient to extend this functional to $\calM(X)\times\calM(Y)$ ``by infinity," namely
$$
W_2^2(\mu,\nu):=
 \infty 
$$
if either $\mu\not\in P(X)$ or $\nu\not\in P(Y)$.

Our work so far can be summarized in terms of Legendre duality these two functionals:
indeed, for any $\mu\in P(X)$,
$$
\begin{aligned}
-W_2^2(\mu,\nu)
&=
\inf_{f\in C(X)}
\Big[
\int_Xf\mu+\int_Y f^\star\nu
\Big]
\cr
&=
-\sup_{f\in C(X)}
\Big[
\langle f,-\mu \rangle -J_\nu(f)
\Big]
=-J_\nu^\star(-\mu),
\end{aligned}
$$
while if $\mu\in \calM(X)$ but $\mu\not\in P(X)$ then
since $(f+C)^\star=f^\star-C$ we see
$\langle f+C,-\mu \rangle -J(f+C)=
\langle f,-\mu \rangle -J_\nu(f)+C(1-\mu(X))$
which can be made arbitrarily large if $\mu(X)\not=1$, which is to say that
$J_\nu^\star(-\mu)= \infty $,
i.e., we have
\begin{equation}
\lb{w22jeq}
\begin{aligned}
W_2^2(\mu,\nu)
=J_\nu^\star(-\mu).
\end{aligned}
\end{equation}  
The following theorem summarizes this and more.
This is the first time the {\it Monge--Amp\`ere operator } 
\begin{equation}
\begin{aligned}
\label{maopeq}
\MA_\nu f:=(\nabla f^\star)_\#\nu
\end{aligned}
\end{equation}
makes its appearance. 

\begin{remark} {\rm
\label{}
When $\nu=dx$ this reduces to the well-known Monge--Amp\`ere operator
$$
\MA f=d\nabla f:=d\frac{\del f}{\del x_1}\w\cdots\w d\frac{\del f}{\del x_n},
$$
since
$$
\int_X F(\nabla f^\star)_\# dx=
\int_Y F\circ(\nabla f^\star) dx
=
\int_Y F\circ(\nabla f)^{-1} dx
=
\int_X F d\nabla f(x)
=
\int_X F \det\nabla f(x).
$$
When $f\in C^2$ then $\MA f=\det\nabla^2 f$.
} \end{remark}

\begin{theorem}
\label{JW2Thm}
$J_\nu$ and $W_2^2(-\,\cdot\,,\nu)$ are convex, lower semicontinuous, and Legendre dual to each other.
$J_\nu$ is Gateaux differentiable and
\begin{equation}
\begin{aligned}
\label{djeq}
dJ_\nu|_f=-\MA_\nu f.
\end{aligned}
\end{equation}  
\end{theorem}

\begin{proof}
\bu First, note that $J_\nu $  is actually continuous: if $C(X)\ni f_j\ra f$ in $C^0$ then $f_j^\star\ra f^\star$ pointwise and hence uniformly \cite[Theorem 10.8]{Rock} so $J$ is continuous. 

\bu Convexity of $J$ is elementary:
$$
\begin{aligned}
J_\nu\Big(\frac{\th+\chi}2\Big)
&=\int_Y\sup_x\Big[\langle x,y \rangle- \frac{\th(x)+\chi(x)}2
\Big]\nu\cr
&\le
\frac12
\int_Y\sup_x[\langle x,y \rangle- \th(x)]\nu+ 
\frac12
\int_Y\sup_x[\langle x,y \rangle- \chi(x)]\nu
\cr
&=
\frac12J_\nu(\th)+\frac12J_\nu(\chi).
\end{aligned}
$$

\bu Convexity of $W_2^2(-\,\cdot\,,\nu)$ then follows from  \eqref{w22jeq}, being the supremum of affine functionals, which also implies lower semi-continuity (the supremum of continuous functions).

\bu Legendre duality was already proven in  \eqref{w22jeq}.

\bu Legendre duality implies that
\begin{equation}
\begin{aligned}
\label{w22jineq}
J_\nu(f)+W_2^2(\mu,\nu)\ge -\langle f,\mu \rangle .
\end{aligned}
\end{equation}
Fix $f$. 
To show Gateaux differentiability and  \eqref{djeq} 
it suffices to show that there exists exactly one $\mu$ for which equality is attained, and that then $\mu=\MA_\nu f$. Define $\mu=:(\nabla f^\star)_\#\nu$.
By Theorem \ref{ftotthm} and Corollary \ref{dualCor}
then $(\id\otimes\nabla f)_\#\mu\in\Pi(\mu,\nu)$ is an optimal transport plan  and
$$
W_2^2(\mu,\nu)
= -\langle f,\mu \rangle-\langle f^\star,\nu\rangle
= -\langle f,\mu \rangle-J_\nu(f).
$$ 
Now, $\mu=(\nabla f^\star)_\#\nu=\MA_\nu f$ by definition. Thus, it suffices to show that this $\mu$ is the only one attaining equality in  \eqref{w22jineq}, i.e., it suffices to show that $J $ is strictly convex.
Suppose $\alpha$ is another such measure, i.e.,
$$
W_2^2(\alpha,\nu)
= -\langle f,\alpha \rangle-\langle f^\star,\nu\rangle.
$$
From Theorem \ref{ftotthm} and Corollary \ref{dualCor}
$(\id\otimes\nabla f)_\#\alpha\in\Pi(\alpha,\nu)$ is an optimal transport plan
and $\alpha=(\nabla f^\star)_\#\nu$ so $\alpha=\mu$.
\end{proof}

\subsection{Rate function for Monge--Amp\`ere }
\label{}

Let  $\be\in\RR$ and let 
$$
\mu_0 \in P(X)
$$
be a fixed reference probability  measure. We are interested in the Monge--Amp\`ere equation
\begin{equation}
\begin{aligned}
\label{maeq}
\MA_\nu f=e^{\be f}\mu_0 \Big/ \int_X e^{\be f}\mu_0.
\end{aligned}
\end{equation}

Define $F_{\be,\nu}:C(X)\ra \RR$ by
\begin{equation}
\begin{aligned}
\label{Feq}
F_{\be,\mu_0,\nu}(\th) : =
\frac1\beta I_{
\mu_0}(\be\th) +  J_\nu(\th).
\end{aligned}
\end{equation}

By \eqref{djeq} and the proof of Lemma  \ref{Entlemma}:
\begin{lemma}
\label{fminlemma}
$F_{\be,\mu_0,\nu}$ is Gateaux differentiable and
\begin{equation}
\begin{aligned}
\label{dfbetaeq}
dF_{\be,\mu_0,\nu}|_\th=
e^{\be f}\mu_0 \Big/ \int_X e^{\be f}\mu_0-\MA_\nu f.
\end{aligned}
\end{equation}
\end{lemma}

Finally, we can define the rate function underlining the  Monge--Amp\`ere equation:
\begin{equation}
\begin{aligned}
\label{geq}
G_{\be,\mu_0,\nu}: =\be W^2_2(\,\cdot\,,\nu)+\Ent(\mu_0,\,\cdot\,)+C,
\end{aligned}
\end{equation}
where $C $  is a constant that will guarantee the function is nonnegative and zero at its minimum.

\begin{prop}
\label{gprop}
Assume that $F_{\be,\mu_0,\nu}$ admits a unique (up to a constant)  minimizer 
$\phi_{\min}$. Then $G_{\be,\mu_0,\nu}$ admits a unique minimizer
$\mu=\MA_\nu\phi_{\min}$. The converse is also true.
\end{prop}

Before going into the proof, let us motivate it with a general observation about Legendre duals in finite dimensions. If
$$
F=f_1+f_2
$$
is the sum of two differentiable strictly convex functions, and $x$ is the unique minimum of $F$ then
$$
G(y):=f_1^\star(y)+f_2^\star(-y)
$$
has a unique minimum at $df_1(x)$.
Indeed, 
$$
df_1(x)= -df_2(x)
$$
while ($G$ is differentiable since $f_1^\star$ and $f_2^\star$ are
by the strict convexity of $f_1,f_2$ \cite[Theorem 26.3] {Rock})
$$
dG(y)=df_1^\star(y)-df_2^\star(-y)
=(df_1)^{-1}(y)-(df_2)^{-1}(-y)
$$
and setting $y=df_1(x)=-df_2(x)$
$$
dG(y)=(df_1)^{-1}(df_1(x))-(df_2)^{-1}(--df_2(x))=x-x=0.
$$
Thus, $y$ is a critical point of the convex function $G$, hence a minimum point.
This is the only minimum point since the proof is reversible: if $dG(\tilde  y)=0$ we get 
$df_1^\star(\tilde  y)=df_2^\star(-\tilde y)$ so 
$df_1(\tilde x)= -df_2(\tilde x)$ for $\tilde x=df_1^\star(\tilde y)$, so then $\tilde x$
is a critical point of $F$, hence a minimum, so then
$x=\tilde x$ since by assumption $x$ was the unique minimum. 
Thus $df_1^\star( y)=df_1^\star(\tilde y)$ implying $y=\tilde y$ 
if $f_1^\star$ is strictly convex, but this follows from differentiability of $f_1$ \cite[Theorem 26.3] {Rock}.

\begin{proof}
Essentially, the conclusion of the finite-dimensional discussion above holds also in our situation by chasing through the definitions and avoiding the use of
\cite[Theorem 26.3] {Rock}. Here goes.

First, by Theorem \ref{JW2Thm} and \eqref{LegBasicProp},
\begin{equation*}
\begin{aligned}
W_2^2(\mu,\nu)+J_\nu(f)\ge -\langle f,\mu \rangle,
\qq \h{equality if and only if $\mu=\MA_\nu f$}.
\end{aligned}
\end{equation*}
Second, by Lemma \ref{Entlemma} and \eqref{LegBasicProp},
\begin{equation*}
\begin{aligned}
\Ent(\mu_0,\mu)+I_{\mu_0}(f)\ge\langle f,\mu\rangle,
\qq \h{equality if and only if $\mu= e^f\mu_0/\int_Xe^f\mu_0$}.
\end{aligned}
\end{equation*}
Let $\phi_{\min}$ be the minimizer of $F_{\be,\mu_0,\nu}$.
By Lemma \ref{fminlemma}
\begin{equation}
\begin{aligned}
\label{fminlemmaeq}
\MA_\nu \phi_{\min}=e^{\be \phi_{\min}}\mu_0 \Big/ \int_X e^{\be \phi_{\min}}\mu_0.
\end{aligned}
\end{equation}
\bu Assume first $\be>0$. Then setting $f=\phi_{\min}$ and $f=\be\phi_{\min}$, respectively,
 in the inequalities above
\begin{equation*}
\begin{aligned}
\label{}
G_{\be,\mu_0,\nu}(\mu)
&=
\be W^2_2(\mu,\nu)+\Ent(\mu_0,\mu)+C,
\cr
&\ge
 -\be\langle \phi_{\min},\mu \rangle-\be J_\nu(\phi_{\min})
-I_{\mu_0}(\be\phi_{\min})+\langle \be\phi_{\min},\mu\rangle
\cr
&=
 -\be J_\nu(\phi_{\min})
-I_{\mu_0}(\be\phi_{\min})
\cr
&=-\be F_{\be,\mu_0,\nu}(\phi_{\min}),
\end{aligned}
\end{equation*}
with equality if and only if $\mu= e^{\be \phi_{\min}}\mu_0/\int_Xe^{\be\phi_{\min}}\mu_0=\MA_\nu \phi_{\min}$. Note that  $-\be F_{\be,\mu_0,\nu}(\phi_{\min})$ is some
constant independent of $\mu$. Thus,  $\mu=\MA_\nu \phi_{\min}$ is the unique
minimizer of $G_{\be,\mu_0,\nu}$.

\bu Assume now that $\be<0$.
Fix $\mu\in \calM(X)$. Let $\phi\in C(X)$ be such that equality hold in \eqref{w22jineq}. Now applying above argument to 
$f=\phi$ and $f=\be\phi$, respectively, gives
\begin{equation*}
\begin{aligned}
\label{}
G_{\be,\mu_0,\nu}(\mu)
\ge-\be F_{\be,\mu_0,\nu}(\phi)
\ge-\be F_{\be,\mu_0,\nu}(\phi_{\min}),
\end{aligned}
\end{equation*}
(the last inequality simply because $\phi_{\min}$ is a minimizer of
$F_{\be,\mu_0,\nu}$ and $\be<0$)
with equality in the first inequality 
if and only if $\mu= e^{\be \phi}\mu_0/\int_Xe^{\be\phi}\mu_0$
and in the second inequality if and only if $\phi= \phi_{\min}$
so overall $\mu= e^{\be \phi_{\min}}\mu_0/\int_Xe^{\be\phi_{\min}}\mu_0
=\MA_\nu\phi_{\min}$ by  \eqref{fminlemmaeq}. 
\end{proof}

\begin{remark} We leave the details for the simpler case $\beta=0$ to the reader
(in this special case the Wasserstein distance does not even appear, and one is basically reduced to Sanov's Theorem (Corollary \ref{SanovCor})). Of course,
one has to also define $F_{0,\mu_0,\nu}$ appropriately by taking the derivative
at $\be=0$ of \eqref{Feq}.
\end{remark}

\section{Moment generating function for Monge--Amp\`ere }
\label{MomentMASec}

Our goal is now to construct a sequence of probability measures on $P(X)$
(i.e., random measures, or elements of $P(P(X))$) whose 
moment generating function (for some normalization) is precisely
$G_{\be,\mu_0,\nu}$.

Naturally, in view of Sanov's Theorem (Corollary \ref{SanovCor}), the entropy term in 
$G_{\be,\mu_0,\nu}$ will come from $\mu_0^{\otimes n}$. To obtain the Wasserstein distance term we will need to multiply the symmetric measure by a symmetric function that captures discrete optimal transport distance.
Here is the key observation \cite[Theorem 3.2]{EllisHavenTurk}.
Let $H_{n^d}:X^{n^d}\ra\RR$ (the reader can basically consider the examples
\eqref{hnpereq} and \eqref{hntropeq} although the next lemma is more general).
Set 
$$
\Gamma_{\be,n}:=\delta^{n^d}_\#\Big( e^{-\be H_{n^d}} \mu_0^{\otimes {n^d}}\Big)
\Big/ Z_{\be,n}\in P(P(X)),
$$
with $Z_{\be,n}:=\int_{X^{n^d}} e^{-\be H_{n^d}} \mu_0^{\otimes {n^d}}$
is the normalizing constant guaranteeing that 
$
\Gamma_{\be,n}
$ is a probability measure; let 
\begin{equation}
\begin{aligned}
\label{Cbeneq}
C_{\be}:=\lim_n\frac1{n^d}\log Z_{\be,n},
\end{aligned}
\end{equation}
where the limit exists and is finite according to Claim \ref{Cbenclaim} below. 

\begin{lemma}
\label{EHnLemma}
Let $E:P(X)\ra\RR$ be continuous and let $H_{n^d}:X^{n^d}\ra\RR$.
Suppose that $\lim_{n\ra \infty }|| H_{n^d}/n^d-E\circ \delta^{n^d}||_{L^ \infty (X^{n^d})}=0$.
\newline
Then LDP$(\Gamma_{\be,n},{n^d})$ with rate function $\be E+\Ent(\mu_0,\,\cdot\,)+C_{\be,n}$,
where 
$$
C_{\be}=-\inf_\mu\big[\be E(\mu)+\Ent({\mu_0},\mu)\big].
$$
\end{lemma}

\begin{remark} {\rm
\label{}
In some sense, the constant $C_{\be}$ has to equal this value if this function is to be a rate function, indeed this
way the infimum is equal to zero, as it must by Remark \ref{zeroremark}.
} \end{remark}

\begin{proof}
This does not seem to follow easily from a moment generating function computation.
Instead, we use the more direct criterion given by Proposition \ref{ldpprop}.
Now  \hbox{$\calX = P(X)$}, with balls taken with respect to the $p$-Wasserstein distance ($p\in[1,\infty)$).
Here we need the fact that when $X$ is compact,
$P(X)$ equipped with the  $p$-Wasserstein distance function is a compact metric space
\cite[\S2]{AmbGig}, \cite{Villani:oldandnew}
(the point is that ``$p$-Wasserstein distance metrizes the weak topology" and that 
$P(X)$ is compact with respect to the weak topology).
We compute
(and apply Claim \ref{Cbenclaim} below), 
\begin{equation*}
\begin{aligned}
\lim _{e\ra 0}\limsup_{n\ra  \infty } \frac1{n^d}\log \Gamma_{\be,n} (B_e (\mu))
&=
\lim _{e\ra 0}\limsup_{n\ra  \infty } 
\frac1{n^d}\log \int_{(\delta^{n^d })^{-1}(B_e (\mu))}e^{-\be H_{n^d}} \mu_0^{\otimes n^d }-C_{\be,n}
\cr
&=
\lim _{e\ra 0}\limsup_{n\ra  \infty } 
\frac1{n^d}\log \int_{(\delta^{n^d })^{-1}(B_e (\mu))}e^{-\be n^d (E\circ\delta^n+o(1))} \mu_0^{\otimes n^d }-C_{\be,n}
\cr
&=
-\be E(\mu)+
\lim _{e\ra 0}\limsup_{n\ra  \infty } 
\frac1{n^d}\log \int_{(\delta^{n^d })^{-1}(B_e (\mu))} \mu_0^{\otimes n^d }-C_{\be,n}
\cr
&=
-\be E(\mu)-\Ent({\mu_0},\mu)-C_{\be,n}
,
\end{aligned}
\end{equation*}
by Corollary \ref{SanovCor} and Proposition \ref{ldpprop}
(remembering the minus sign in the latter). 
Similarly for the liminf. Applying Proposition \ref{ldpprop} again, we are done.
\end{proof}

\begin{claim}
\label{Cbenclaim}
The limit  \eqref{Cbeneq} exists and is finite. In fact, 
$$
C_{\be,n}=-\inf_\mu\big[\be E(\mu)+\Ent({\mu_0},\mu)\big].
$$
\end{claim}

\begin{proof}
By our previous computation the limit is bounded below, indeed, for any
$e>0$ and any $\mu$,
\begin{equation*}
\begin{aligned}
\label{}
\liminf_n\frac1{n^d}\log Z_{\be,n}
&=
\liminf_n\frac1{n^d}\log
\int_{X^{n^d}} e^{-\be H_{n^d}} \mu_0^{\otimes {n^d}}
\cr
&\ge
\liminf_n\frac1{n^d}\log
\int_{(\delta^{n^d })^{-1}(B_e (\mu))} e^{-\be H_{n^d}} \mu_0^{\otimes {n^d}},
\end{aligned}
\end{equation*}
so for every $\mu$,
\begin{equation*}
\begin{aligned}
\label{}
\liminf_n\frac1{n^d}\log Z_{\be,n}
&\ge
-\be E(\mu)-\Ent({\mu_0},\mu)
\cr
&\ge
\sup_\mu\big[-\be E(\mu)-\Ent({\mu_0},\mu)\big]
\cr
&=-\inf_\mu\big[\be E(\mu)+\Ent({\mu_0},\mu)\big]
,
\end{aligned}
\end{equation*}
and so,
\begin{equation*}
\begin{aligned}
\label{}
\limsup_n\frac1{n^d}\log Z_{\be,n}
&\ge
-\inf_\mu\big[\be E(\mu)+\Ent({\mu_0},\mu)\big]
.
\end{aligned}
\end{equation*}
Finally, by compactness of $P(X)$, fix $e>0$ and cover the space with finitely-many
balls $B_e(\mu_1),\ldots,B_e(\mu_k)$. Then, of course,
$$
\begin{aligned}
Z_{\be,n}
&=
\int_{P(X)} (\delta^{n^d })_{\#}\Big(e^{-\be H_{n^d}} \mu_0^{\otimes {n^d}}\Big)
\cr
&=
\int_{(\delta^{n^d })^{-1}(P(X))} e^{-\be H_{n^d}} \mu_0^{\otimes {n^d}}
\cr
&\le
\sum_{j=1}^k\int_{(\delta^{n^d })^{-1}(B_e(\mu_j))} 
e^{-\be H_{n^d}} \mu_0^{\otimes {n^d}}
\cr
&\le
k\sup_j\int_{(\delta^{n^d })^{-1}(B_e(\mu_j))} 
e^{-\be H_{n^d}} \mu_0^{\otimes {n^d}},
\end{aligned}
$$
so
$$
\begin{aligned}
\liminf_n\frac1{n^d}\log Z_{\be,n}
&\le
\liminf_n\frac1{n^d}\log k
+\liminf_n\frac1{n^d}\log \sup_j\int_{(\delta^{n^d })^{-1}(B_e(\mu_j))} 
e^{-\be H_{n^d}} \mu_0^{\otimes {n^d}}
\cr
&=
\liminf_n\frac1{n^d}\log \sup_j\int_{(\delta^{n^d })^{-1}(B_e(\mu_j))} 
e^{-\be H_{n^d}} \mu_0^{\otimes {n^d}}
\cr
&=
\sup_j\big[-\be E(\mu_j)-\Ent({\mu_0},\mu_j)\big]
\cr
&=
-\inf_j\big[\be E(\mu_j)+\Ent({\mu_0},\mu_j)\big]
\cr
&\le
-\inf_\mu\big[\be E(\mu)+\Ent({\mu_0},\mu)\big]
.
\end{aligned}
$$
Similarly,
$$
\begin{aligned}
\limsup_n\frac1{n^d}\log Z_{\be,n}
\le
-\inf_\mu\big[\be E(\mu)+\Ent({\mu_0},\mu)\big]
,
\end{aligned}
$$
so we conclude
$C_{\be,n}=\lim_n\frac1{n^d}\log Z_{\be,n}$ exists
and equals $-\inf_\mu\big[\be E(\mu)+\Ent({\mu_0},\mu)\big]$.
\end{proof}
\subsection{Finite-dimensional approximations of Wasserstein distance}
\label{FinDimWasSubSec}
In view of Lemma \ref{EHnLemma}, Proposition \ref{gprop} (and  \eqref{geq}), 
it remains for us to construct functions $\{H_n\}$ that approximate 
the (pull-back under the empirical map of the) Wasserstein distance $W^2_2(\,\cdot\,,\nu)$. We will do this in the special case $\nu=dx$.

Let $n\in\NN$. There are $n^d$ $1/n$-lattice points of the cube $[0,1]^d$,
and we denote them by $p_1,\ldots,p_{n^d}$. Set
$$
\phi^{(n)}_i(x):=\sum_{m\in\ZZ^d} e^{-n|x-p_i-m|^2}.
$$ 
A sort of ``theta function" for the real torus 
$$
\TT:=\RR^d/\ZZ^d.
$$
There are two sorts of symmetric functions on $\TT^{n^d}$ one may cook up from the $\phi_i$'s. First, consider the matrix
$$
\Phi(x_1,\ldots,x_n):=[\phi^{(n)}_i(x_j)]_{i,j=1}^{n^d}.
$$
Which functions $f$ of $\Phi(x_1,\ldots,x_n)$ are invariant under permutations, i.e.,
satisfy $$
f(\Phi(x_1,\ldots,x_n))=f(\Phi(x_{\sigma(1)},\ldots,x_{\sigma(n)}))?
$$
 In other words, which functions of a matrix are invariant under permutations of rows? Note that the determinant is only invariant up to a sign. However, the permanent is
fully invariant.
First, 
\begin{equation}
\begin{aligned}
\label{hnpereq}
H_n(x_1,\ldots,x_n):=-\frac1n\log\per\Phi(x_1,\ldots,x_n).
\end{aligned}
\end{equation}
Here,
$$
\per A:=\sum_{\sigma\in S_{n^d}}\prod_{i=1}^{n^d}A_{i\sigma(i)}.
$$
Second,
\begin{equation}
\begin{aligned}
\label{hntropeq}
H_n(x_1,\ldots,x_n):=-\frac1n\log
\tsper\Phi(x_1,\ldots,x_n)
.
\end{aligned}
\end{equation}
Here, the {\it semi-tropical permanent } is obtained from the permanent by replacing summation by supremum,
$$
\tsper A:=\sup_{\sigma\in S_{n^d}}\prod_{i=1}^{n^d}A_{i\sigma(i)}.
$$

\begin{lemma}
\label{HnWLemma}
For both \eqref{hnpereq}
and \eqref{hntropeq},
$\lim_{n\ra \infty }|| H_n/n^d-W_2^2\big(dx,\delta^n(\,\cdot\,)\big)||_{L^ \infty (X^n)}=0$.
\end{lemma}

\begin{proof}
Since $\delta^{n^d} (p_1,\ldots , p_{n^d})\ra dx$ weakly (the points are dense and uniformly distributed) then in view of Claim \ref{w2claim} below, it suffices to show that 
$$\lim_{n\ra \infty }|| H_n/n^d- W_2^2\big(\delta^{n^d} (p_1,\ldots , p_{n^d}),\delta^{n^d}(\,\cdot\,)\big)||_{L^ \infty (X^{n^d})}=0.$$
This is a nice simplification since we have an explicit formula for the Wasserstein distance on the image of the empirical map! Indeed \cite[p. 5]{Villani},
$$
W_2^2\big(\delta^{n^d} (p_1,\ldots , p_{n^d}),\delta^{n^d}(x_1,\ldots, x_{n^d})\big)=
\inf_{\sigma\in S_{n^d}}\sum d(p_i,x_{\sigma(i)})^2.
$$
It is now a simple exercise to complete the proof using Claims 
\ref{fsigmaclaim} and \ref{phiclaim} below.
\end{proof}

\begin{claim}
\label{w2claim}
Let $M $ be compact manifold.  Let $x_1,\ldots, x_k\in M$
and $p_1,\ldots, p_k\in M$. Suppose that $\delta^k (p_1,\ldots , p_k)\ra \nu$ weakly.  Then 
$$
\lim_k
||W_2^2\big(\delta^k (p_1,\ldots , p_k),\delta^k(\,\cdot\,)\big)
-W_2^2\big(\nu,\delta^k(\,\cdot\,)\big)||_{L^ \infty (X^k)}=0.
$$
\end{claim}

\begin{proof}
Wasserstein distance is a distance function (see \cite{Villani,Villani:oldandnew}), hence,
$$
\big|W_2\big(\delta^k (p_1,\ldots , p_k),\delta^k(x_1,\ldots, x_k)\big)
-W_2\big(\nu,\delta^k(x_1,\ldots, x_k)\big)
\big|
\le
W_2\big(\delta^k (p_1,\ldots , p_k),\nu\big),
$$
with the right-hand side independent of $x_1,\ldots, x_k$.
On a compact manifold weak convergence implies convergence in the Wasserstein distance, hence the right-hand side converges to zero as $k$ tends to infinity.
Finally,
\begin{equation*}
\begin{aligned}
\label{}
&\big|W^2_2\big(\delta^k (p_1,\ldots , p_k),\delta^k(x_1,\ldots, x_k)\big)
-W^2_2\big(\nu,\delta^k(x_1,\ldots, x_k)\big)
\big|
\cr
&\qq
\le\big(W_2(\delta^k (p_1,\ldots , p_k),\delta^k(x_1,\ldots, x_k))
+W_2\big(\nu,\delta^k(x_1,\ldots, x_k)\big)
\big)
W_2\big(\delta^k (p_1,\ldots , p_k),\nu\big)
\cr
&\qq
\le
\big( 
2W_2\big(\delta^k (p_1,\ldots , p_k),\delta^k(x_1,\ldots, x_k)\big)+o(1)\big)
W_2(\delta^k (p_1,\ldots , p_k),\nu)
\cr
&\qq
\le
\big( 
2\frac1k\sum d(x_i,p_i)^2+o(1)\big)
W_2\big(\delta^k (p_1,\ldots , p_k),\nu\big)
\cr
&\qq
\le
\big( 
C(M)+o(1)\big)
W_2\big(\delta^k (p_1,\ldots , p_k),\nu\big)
,
\end{aligned}
\end{equation*}
since compactness implies the diameter is bounded. This concludes the proof.\end{proof}

\begin{claim}
\label{fsigmaclaim}
Let $F:S_{n^d}\ra(0, \infty )$. Then
$$
\frac 1{n^{d+1}}\log \sup_{\sigma}F(\sigma)
=
\frac 1{n^{d+1}}\log \sum_{\sigma}F(\sigma)+o(1).
$$
\end{claim}

\begin{proof}
Of course,
$$
\frac 1{n^{d+1}}\log \sup_{\sigma}F(\sigma)
\le
\frac 1{n^{d+1}}\log \sum_{\sigma}F(\sigma).
$$
Conversely,
$$
\frac 1{n^{d+1}}\log \sum_{\sigma}F(\sigma)
\le
\frac 1{n^{d+1}}\log n^d\sup_{\sigma}F(\sigma).
$$
and by Stirling $\frac 1{n^{d+1}}\log n^d=o(1)$.
\end{proof}

\begin{claim}
\label{phiclaim}
$-\frac1n\log\phi^{(n)}_i(x)=d(p_i,x)^2+o(1)$.
\end{claim}

\begin{proof}
By definition,
$$
d(p_i,x)^2=\inf_{m\in\ZZ^d} |x-p_i-m|^2.
$$
Now, of course,
$$
\frac1n\log\phi^{(n)}_i(x)=
\frac1n\log\sum_{m\in\ZZ^d} e^{-n|x-p_i-m|^2}
\ge
\frac1n\log\sup_{m\in\ZZ^d} e^{-n|x-p_i-m|^2}
=-d(p_i,x)^2.
$$
Conversely, for every $\epsilon > 0 $ there exists $C, R >0$ 
such that 
$$\sum_{m\in\ZZ^d} e^{-n|x-p_i-m|^2}\le 
\sum_{m\in B_R(0)\cap \ZZ^d} e^{-n|x-p_i-m|^2}+\eps
\le C\sum_{m\in B_R(0)\cap \ZZ^d} e^{-n|x-p_i-m|^2}.
$$
Assuming, without loss of generality, that
$\sup_{m\in\ZZ^d} e^{-n|x-p_i-m|^2}$ is obtained in $B_R(0)$, we have
$$
\sum_{m\in B_R(0)\cap \ZZ^d} e^{-n|x-p_i-m|^2}\le 
CR^d\sup_{m\in\ZZ^d} e^{-n|x-p_i-m|^2},
$$
so
\begin{equation*}
\begin{aligned}
\label{}
\frac1n\log\phi^{(n)}_i(x)
&=
\frac1n\log\sum_{m\in\ZZ^d} e^{-n|x-p_i-m|^2}
\cr
&\le
\frac1n\log C +\frac1n\log\sup_{m\in B_R(0)\cap \ZZ^d} e^{-n|x-p_i-m|^2}
\cr
&\le
o(1)+\frac1n\log CR^d\sup_{m\in\ZZ^d} e^{-n|x-p_i-m|^2}
\cr
&=o(1)+
\frac1n\log\sup_{m\in\ZZ^d} e^{-n|x-p_i-m|^2}
\cr
&=o(1)-d(p_i,x)^2,
\end{aligned}
\end{equation*}
where $o(1)$ depends on $R$, but goes to zero as $n $ tends to infinity (for $R$ fixed). Letting $n $ tend to infinity concludes the proof.
\end{proof}

Finally, we obtain the following theorem due to Berman \cite[Theorem 1.1]{Berm}
and Hultgren \cite[Theorem 3.2]{Hult}.

\begin{theorem}
\label{hultthm}
For both \eqref{hnpereq}
and \eqref{hntropeq}, LDP$(\Gamma_{\be,n},n^d)$ with rate function $G_{\be,\mu_0,dx}$.
The set $G_{\be,\mu_0,dx}^{-1}(0)$ is a singleton precisely when  \eqref{maeq} has a unique solution.
\end{theorem}

\begin{proof}
The first statement follows from Lemmas \ref{EHnLemma} and \ref{HnWLemma}.
The second statement follows from Proposition \ref{gprop}. 
\end{proof}

\subsection{An alternative proof---zero temperature approach}
\label{AltProofSubSec}

We were lucky enough to find finite-dimensional approximations to the main functional we were interested in (for all $\beta$ at once). Here is an alternative approach, due to Berman in the permanental/Monge--Amp\`ere setting, which allows to `reverse engineer' the main functional by computing the limit of the moment generating functions when $\beta=\beta_n=n\rightarrow \infty$. 
This is an easier task because in this ``zero-temprature limit'' the entropy contribution disappears.
This is a standard method in the field and its benefit in this setting is that when writing out the moment generating functions explicitly, the symmetry in the Hamiltonians can be exploited to reduce much of the complexity. Once an explicit formula for the limit is attained \cite[Proposition 5.3]{Berm}, 
\cite[Lemma 3.8]{Hult}, the G\"artner--Ellis Theorem can be invoked to deduce an LDP for this ``zero-temperature" case (stated as a part of Theorem 1.1 in Berman's paper \cite{Berm} and as Theorem 3.6 in Hultgren's article \cite{Hult}). This LDP can then be used to deduce Lemma 8.4 above (corresponding to Lemma 4.9 in \cite{Berm} and Lemma 3.14 in \cite{Hult}), after which Theorem 8.8 is proved as in the previous section. 
This original approach of Berman and Hultgren to proving Theorem \ref{hultthm} also can be made to work in the case we are no longer on the torus, but rather on a non-compact manifold, as in the toric setting. The reason we chose the proof 
we presented above is that we found it slightly more pedagogical to directly deal with all $\beta$ at once and avoid this extra use of the Gartner--Ellis theorem.

We will now explain the main part of the alternative argument for 
Theorem \ref{hultthm}, that as just mentioned, was the original proof. 
Namely, how to prove the LDP when $\beta=\beta_k=k\rightarrow \infty$. Here is the main observation:

\begin{lemma}
\label{ldpHnLemma}

Let $H_{n}:X^{n^d}\ra\RR$ be given by \eqref{hnpereq}
or \eqref{hntropeq}.
Set 
$$
\Gamma_{n,n}:=\delta^{n^d}_\#\Big( e^{-nH_{n}} \mu_0^{\otimes {n^d}}\Big)
\Big/Z_{n,n}\in P(P(X)).
$$
Then LDP$(\Gamma_{n,n},n^{d+1})$ with rate function $W_2^2(dx,\,\cdot\,)$.
\end{lemma}

\begin{proof}
This time, we can use Theorem \ref{GEThm}. 
By Claim \ref{Znnclaim} below, $\lim\frac 1{n^{d+1}}
\log Z_{n,n}=0$. Thus, assuming  \eqref{hnpereq}, the moment generating function 
simplifies as follows,
\begin{equation*}
\begin{aligned}
\label{}
p (\theta) 
&=
\lim\frac 1{n^{d+1}}
\log\int^{}_{P(X)}e^{n^{d+1} \langle \theta,\nu \rangle }\Gamma_{n,n}(\nu)  
\cr
&=
\lim\frac 1{n^{d+1}}
\log\int^{}_{X^{n^d}}e^{n^{d+1} \langle \theta,\delta^{n^d}(\,\cdot\,) \rangle }
\Big(e^{-nH_n}\mu_0^{\otimes n^{d}}\Big)
\cr
&=
\lim\frac 1{n^{d+1}}
\log\int^{}_{X^{n^d}}e^{n^{d+1} \langle \theta,\delta^{n^d}(x_1,\ldots,x_{n^d})\rangle }
e^{-n H_{n}(x_1,\ldots,x_{n^d})}\mu_0(x_1)\otimes\cdots\mu_0(x_{n^d})
\cr
&=
\lim\frac 1{n^{d+1}}
\log\int^{}_{X^{n^d}}e^{n^{d+1} n^{-d}\sum_{i = 1}^{n^d}\theta(x_i)}
e^{-n H_{n}(x_1,\ldots,x_{n^d})}\mu_0(x_1)\otimes\cdots\mu_0(x_{n^d})
\cr
&=
\lim\frac 1{n^{d+1}}
\log\int^{}_{X^{n^d}}e^{n\sum_{i = 1}^{n^d}\theta(x_i)}
\sum_\sigma\prod \phi^{(n)}_i(x_{\sigma(i)})\mu_0(x_1)\otimes\cdots\mu_0(x_{n^d})
\cr
&=
\lim\frac 1{n^{d+1}}
\log\int^{}_{X^{n^d}}
\sum_\sigma\prod \Big[ e^{n\theta(x_{\sigma(i)})} \phi^{(n)}_i(x_{\sigma(i)})\Big]\mu_0(x_1)\otimes\cdots\mu_0(x_{n^d})
\cr
&=
\lim\frac 1{n^{d+1}}
\log\int^{}_{X^{n^d}}
n^d!\prod \Big[ e^{n\theta(x_{i})} \phi^{(n)}_i(x_i)\Big]\mu_0(x_1)\otimes\cdots\mu_0(x_{n^d})
\cr
&=
\lim\frac 1{n^{d+1}}
\log\int^{}_{X^{n^d}}
\prod \Big[ e^{n\theta(x_{i})} \phi^{(n)}_i(x_i)\Big]\mu_0(x_1)\otimes\cdots\mu_0(x_{n^d})
\cr
&=
\lim\frac 1{n^{d+1}}
\log
\prod\Big[
\int^{}_{X} e^{n\theta} \phi^{(n)}_i\mu_0\Big]
\cr
&=
\lim\frac 1{n^{d}}
\sum_{i=1}^{n^d}
\frac1 n\log\Big[
\int^{}_{X} e^{n\theta} \phi^{(n)}_i\mu_0\Big].
\end{aligned}
\end{equation*}
By Claim \ref{phiclaim}, 
$-\frac1n\log\phi^{(n)}_i(x)=d(p_i,x)^2+o(1)$.
and by Claim \ref{Legasympclaim} below we thus have
\begin{equation*}
\begin{aligned}
\label{}
p (\theta) 
&=
\lim\frac 1{n^{d}}
\sum_{i=1}^{n^d}\big[
(-\th)^\star(p_i)+o(1)\big]=
\lim\big\langle \delta^{n^d} (p_1,\ldots , p_{n^d}),(-\th)^\star\big\rangle
=\int_X((-\th)^\star dx,
\end{aligned}
\end{equation*}
since $\delta^k (p_1,\ldots , p_k)\ra dx$.
Thus, by Theorems \ref{GEThm} and \ref{JW2Thm} we are done.
\end{proof}

\begin{claim}
\label{Znnclaim}
$\lim\frac 1{n^{d+1}}
\log Z_{n,n}=0$.
\end{claim}

In fact, we will give a rate of decay, 
$\frac 1{n^{d+1}}
\log Z_{n,n}=O(1/n)$.

\begin{remark} {\rm
\label{}
It actually suffices to show that $\lim\frac 1{n^{d+1}}
\log Z_{n,n}$ exists---it then must be zero:
by Theorem \ref{GEThm}, once we have a large deviation principle and we know that the rate function is $W_2^2(dx,\,\cdot\,)$ up to a constant, then we can determine that constant by the fact that the infimum of the rate function must be zero
(Remark \ref{zeroremark}). Since
$\inf W_2^2(dx,\,\cdot\,)=0$ (attained for $dx$), we get the constant must be zero.
At any rate, we will prove Claim \ref{Znnclaim} directly. 
} \end{remark}

\begin{remark} {\rm
\label{}
In fact, here is a quick proof: $\lim_n\frac 1{n^{d+1}}
\log Z_{n,n}=p(0)$, which, by the previous computation,
equals $\int 0^\star=0$.
} \end{remark}

\begin{proof}
We compute,
\begin{equation*}
\begin{aligned}
\label{}
\lim\frac 1{n^{d+1}}
\log Z_{n,n}
&=
\lim\frac 1{n^{d+1}}
\log\int^{}_{X^{n^d}}\per[\phi^{(n)}_i(x_j)]\mu_0^{\otimes n^{d}}
\cr
&=
\lim\frac 1{n^{d+1}}
\log\int^{}_{X^{n^d}}
\sum_\sigma\prod_{i=1}^{n^d}  \phi^{(n)}_i(x_{\sigma(i)})\mu_0(x_1)\otimes\cdots\mu_0(x_{n^d})
\cr
&=
\lim\frac 1{n^{d+1}}
\log\int^{}_{X^{n^d}}
\sum_\sigma\prod_{i=1}^{n^d}  \phi^{(n)}_i(x_i)\mu_0(x_1)\otimes\cdots\mu_0(x_{n^d})
\cr
&=
\lim\frac 1{n^{d+1}}
\log\int^{}_{X^{n^d}}
n^d!\prod   \phi^{(n)}_i(x_i)\mu_0(x_1)\otimes\cdots\mu_0(x_{n^d})
\cr
&=
\lim\frac 1{n^{d+1}}
\log\int^{}_{X^{n^d}}
\prod   \phi^{(n)}_i(x_i)\mu_0(x_1)\otimes\cdots\mu_0(x_{n^d})
\cr
&=
\lim\frac 1{n^{d+1}}
\log
\prod
\int^{}_{X}  \phi^{(n)}_i\mu_0
\cr
&=
\lim\frac 1{n^{d}}
\sum_{i = 1}^ {n^d}\frac 1n\log||\phi^{(n)}_i||_{L^1(\mu_0)}
\le C_n,
\end{aligned}
\end{equation*}
where $C_n:=\sup_{i=1,\ldots n^d} \frac 1n\log||\phi^{(n)}_i||_{L^1(\mu_0)}$.
Now, it remains to estimate $C_n$.
By Claim \ref{phiclaim}, 
$-\frac1n\log\phi^{(n)}_i(x)=d(p_i,x)^2+o(1)$, so
$$
\begin{aligned}
\frac 1n\log||\phi^{(n)}_i||_{L^1(\mu_0)}
&\le
\frac 1n\log||e^{-n(d(p_i,x)^2+o(1))}||_{L^1(\mu_0)}=O(1/n).
\end{aligned}
$$
Since $i $  was arbitrary, $C_n=O(1/n)$ and we are done.
\end{proof}

\begin{claim}
\label{Legasympclaim}
$\lim_{k\ra \infty } \frac 1k\log \int e^{k(d(x,y)^2-f(x))}dx=f^\star(y)$.
\end{claim}
\begin{proof}
Of course, $\lim_{k\ra \infty }||F||_{L^k(X,\mu)}=||F||_{L^\infty(X)}$ for
continuous $F$ and compact $X$ and probability $\mu$. By definition, 
$\sup_x[d(x,y)^2-f(x)]=f^\star(y)$. So,
$$
\begin{aligned}
\lim_{k\ra \infty } \frac 1k\log \int e^{k(d(x,y)^2-f(x))}dx
&=
\lim_{k\ra \infty } \log ||e^{d(x,y)^2-f(x)}||_{L^k(dx)}
\cr
&=
\log ||e^{d(x,y)^2-f(x)}||_{L^\infty}
\cr
&=
||d(x,y)^2-f(x)||_{L^\infty}
\cr
&=
f^\star(y),
\end{aligned}
$$
as desired.
\end{proof}

The alternative proof of Theorem \ref{hultthm} (that is actually the original proof in \cite{Hult}) is now a consequence. The point is that once we know there is a large deviation principle for $\beta\ra\infty$ we can use Proposition \ref{ldpprop}
and Sanov's Corollary \ref{SanovCor} to deduce the convergence in Lemma \ref{HnWLemma} in an argument which provide a formal converse of Lemma \ref{EHnLemma} above, valid in the $\beta\ra\infty$ case (see Lemma 4.9 in \cite{Berm} or the proof of Theorem 3.2 in \cite{Hult}). After that we get the LDP in Theorem \ref{hultthm} by applying Lemma 
\ref{HnWLemma} as above.

\section*{Acknowledgments}

Many thanks to R. Berman, J. Hultgren, O. Zeitouni, and S. Zelditch for many discussions
on these topics and to them and to the referees for many detailed 
comments/corrections on this manuscript, 
to them and I. Cheltsov for the encouragement to write these notes,
and to J. Hultgren also for enlightening lectures at UMD on \cite{Hult} in September 2017.
Finally, it is a real privilege to dedicate this article to Steve Zelditch whose theorems and vision have been nothing short of pioneering both in complex geometry and in its
relations to probability and microlocal analysis, and whose guidance, generosity, and encouragement
have been invaluable in my own pursuits over the years. 
This work was supported by NSF grants DMS-1515703,1906370,
 and a Sloan Research Fellowship.

\def\listing#1#2#3{{\sc #1}:\ {\it #2}, \ #3.}

\vspace{.5 cm}
\noindent

{\sc University of Maryland}

{\tt yanir@alum.mit.edu}

\end{document}